\documentclass[12 pt,english,oneside]{amsart}
\usepackage[top=1.26in, bottom=1.26in, left=1.33in, right=1.33in]{geometry}
\usepackage[latin1]{inputenc}
\usepackage[T1]{fontenc}
\usepackage[normalem]{ulem}
\usepackage{verbatim} 
\usepackage{graphicx}
\usepackage[all]{xy}
\usepackage{amsfonts}
\usepackage{amsmath}
\usepackage{amssymb}
\usepackage{graphicx,epsfig}
\usepackage{amsthm}
\usepackage{amscd}
\usepackage{dsfont}
\usepackage{fancyhdr}

\newcommand{\g}{\mathfrak{g}}
\newcommand{\D}{\displaystyle}
\newcommand{\ra}{\rightarrow}

\newcommand{\ve}{\varepsilon}
\newcommand{\vp}{\varphi}

\newcommand{\ts}{\otimes}
\newcommand{\s}{\sigma}

\newcommand{\mn}{\mathbb{N}}
\newcommand{\mc}{\mathbb{K}}

\newcommand{\wt}{\widetilde}

\newcommand{\om}{\omega}
\newcommand{\ad}{\operatorname*{ad}}
\newcommand{\Hoch}{\operatorname*{Hoch}}
\newcommand{\id}{\operatorname*{id}}
\newcommand{\HH}{\operatorname*{H}}
\newcommand{\gr}{\operatorname*{gr}}

\newtheorem{definition}{Definition}
\newtheorem{proposition}{Proposition}
\newtheorem{theorem}{Theorem}
\newtheorem{lemma}{Lemma}
\newtheorem{corollary}{Corollary}
\newtheorem{remark}{Remark}

\title{A Borel-Weil-Bott type theorem of quantum shuffle algebras}
\author{Xin FANG}

\begin{document}
\maketitle

\begin{abstract}
We prove in this paper a Borel-Weil-Bott type theorem for the coHochschild homology of a quantum shuffle algebra associated with quantum group datum taking coefficients in some well-chosen bicomodules, which can be looked as an analogue of equivariant line bundles over flag varieties in the non-commutative case.
\end{abstract}

\section{Introduction}
\subsection{Motivations}
One of the central problems of representation theory is the construction of all irreducible or indecomposible representations of a given group or (associative) algebra. In the framework of complex compact Lie groups, there are in general two systematical ways to realize the finite dimensional ones as functions on the corresponding Lie groups: the Peter-Weyl theorem and the Borel-Weil-Bott theorem. The former decomposes the algebra of square integrable functions on a compact Lie group $G$ as a Hilbert direct sum of endomorphism rings of all finite dimensional irreducible representations; the latter views such a representation as global sections of some equivariant line bundle $\mathcal{L}_\lambda$ over the flag variety $G/B$ where $B$ is a fixed Borel subgroup.
\par
Around 1985, Drinfel'd and Jimbo constructed the quantum group $U_q(\g)$ as a deformation of the ordinary enveloping algebra $U(\g)$ associated to a symmetrizable Kac-Moody Lie algebra $\g$. As a quotient of a Drinfel'd double, $U_q(\g)$ is a quasi-triangular Hopf algebra: as a consequence, the category of finite dimensional $U_q(\g)$-modules is braided; this can be applied to the construction of braid group representations and then explicit solutions of the Yang-Baxter equation.
\par
When $q$ is not a root of unity, the similarity between representation theories of $U_q(\g)$ and $U(\g)$ can be explained as a quantization procedure of representations. This phenomenon evokes us to develop necessary tools and frameworks to generalize the Peter-Weyl and Borel-Weil-Bott theorem to quantum groups. A mock version of the former is partially achieved and clarified in a series of work due to Joseph-Letzter \cite{JL94} and Caldero \cite{Cal93}.
\par
An analogue of Borel-Weil-Bott theorem for quantum groups is obtained in an earlier work of Anderson-Polo-Wen \cite{APW91} by viewing the negative part $U_q^-(\g)$ of $U_q(\g)$ as the corresponding flag variety and a character as a line bundle over it. They finally generalized the whole theorem using techniques coming from the representation theory of algebraic groups.
\par
In this paper, we would like to provide another point of view on generalizing the Borel-Weil-Bott theorem to quantum groups. Compared with the approach of Anderson-Polo-Wen (\textit{loc.cit}), our construction has following advantages:
\begin{enumerate}
\item The construction is functorial: no explicit coordinate and character are needed at the very beginning;
\item We may start with a general quantum shuffle algebra (it is well-known that the negative part of a quantum group is a very particular kind of such algebras).
\end{enumerate}

\subsection{Known results}
Apart from the work of Anderson-Polo-Wen, the Hochschild cohomology of algebras associated to quantum groups are also studied in the work of Ginzburg-Kumar \cite{Gin93}: they computed the Hochschild cohomology of the (strictly) negative parts of the restricted quantum groups with coefficients in regular bimodules and one-dimensional weight modules. Methods therein are motivated by those appeared in the study of the representation theory of algebraic groups.

\subsection{Quantum shuffle algebras}
The tensor algebra $T(V)$ associated to a vector space $V$ is a graded Hopf algebra with concatenation multiplication and shuffle comultiplication. Its graded dual $T(V)^{\ast,gr}$ is also a graded Hopf algebra, where the multiplication is given by the shuffle product and comultiplication is the deconcatenation. A quantum shuffle algebra is almost the latter with the replacement of the symmetric group by an associated braid group, where the braiding arises from a Yetter-Drinfel'd module structure.
\par
To be more precise, we start with a Hopf algebra $H$ and an $H$-Hopf bimodule $M$; the right coinvariant space $V=M^{coR}$ admits an $H$-Yetter-Drinfel'd module structure and thus a braiding $\s:V\ts V\ra V\ts V$. The construction of the shuffle algebra with this braiding gives the quantum shuffle algebra $S_\s(V)$ as a subalgebra \cite{Ros98}.

\subsection{Main results}
Suppose that $\g$ is a finite dimensional simple Lie algebra. The main construction of the analogous object of the equivariant line bundle in our framework is given by imposing an element in the set of right coinvariants $M^{coR}$ to build another quantum shuffle algebra $S_{\widetilde{\s}}(W)$, which is automatically a Hopf bimodule over $S_\s(V)$ and the "line bundle" finds itself as a sub-Hopf bimodule $S_{\widetilde{\s}}(W)_{(1)}$ in $S_{\widetilde{\s}}(W)$ by considering a gradation in it.
\par
With these constructions, the main purpose of this paper can be listed as follows:
\begin{enumerate}
\item Explain a non-commutative version of the line bundle over the flag variety $G/B$ and present a Borel-Weil-Bott type theorem in this framework: this is done by considering the coHochschild homology of the quantum shuffle algebra $S_\s(V)$ with coefficient in the $S_\s(V)$-bicomodule $S_{\widetilde{\s}}(W)_{(1)}$ defined above and the main theorem is:\\
\textbf{Theorem A}. The coHochschild homology groups of $S_\s(V)$ with coefficient in the $S_\s(V)$-bicomodule $S_{\wt{\s}}(W)_{(1)}$ are:
\begin{enumerate}
\item If $q$ is not a root of unity and $\lambda\in\mathcal{P}_{+}$, as $U_q(\g)$-modules:
$${\Hoch}^n(S_\s(V),S_{\wt{\s}}(W)_{(1)})=\left\{\begin{matrix}L(\lambda) & n=0;\\ 
0, & n\neq 0.\end{matrix}\right.
$$
\item If $q^l=1$ is a primitive root of unity and $\lambda\in\mathcal{P}_{+}^l$, as $U_q(\g)$-modules:
$${\Hoch}^n(S_\s(V),S_{\wt{\s}}(W)_{(1)})=\left\{\begin{matrix}L(\lambda) & n=0;\\ 
\wedge^n(\mathfrak{n_-}), & n\geq 1.\end{matrix}\right.
$$
where $\mathfrak{n}_-$ is identified with the negative part of the Lie algebra $\g$.
\end{enumerate}

\item Explain how this approach can be generalized to the "line bundle" of higher degree: that is to say, there is a family of $S_\s(V)$-bicomodules $S_{\widetilde{\s}}(W)_{(n)}$ such that the coHochschild homology group of degree $0$ for $S_\s(V)$ with coefficient in these bicomodules can be found as a sum of irreducible representations. More precisely, we will prove the following theorem in degree two:\\
\textbf{Theorem B}. Let $q\in \mc^\ast$ not be a root of unity and $\lambda\in\mathcal{P}_{+}$ be a dominant weight.
\begin{enumerate}
\item If for any $i\in I$, $(\lambda,\alpha_i^\vee)\neq 1$, then as $U_q(\g)$-modules,
$${\Hoch}^n(S_\s(V),S_{\wt{\s}}(W)_{(2)})=\left\{\begin{matrix}L(\lambda)\ts L(\lambda) & n=0;\\ 
0, & n\neq 0.\end{matrix}\right.
$$
\item If $J$ is the subset of $I$ containing those $j\in I$ such that $(\lambda,\alpha_j^\vee)=1$, then as $U_q(\g)$-modules,
$${\Hoch}^n(S_\s(V),S_{\wt{\s}}(W)_{(2)})=\left\{\begin{matrix}(L(\lambda)\ts L(\lambda))/\D\bigoplus_{j\in J}L(2\lambda-\alpha_j) & n=0;\\ 
0, & n\neq 0.\end{matrix}\right.
$$
\end{enumerate}

\end{enumerate}

\subsection{Constitution of this paper}
The organization of this paper is as follows:
\par
We start in Section \ref{Sec6.2} with a recollection on quantum shuffle algebras and constructions around them; the (negative parts of) quantum groups are then given as an example. Section \ref{Section:mainconstruction} is devoted to giving the main construction of the $S_\s(V)$-Hopf bimodule $S_{\wt{\s}}(W)_{(n)}$ and a theorem of Rosso. CoHochschild homology is recalled in Section \ref{Sec6.5}, moreover we discuss the module and comodule structures on these groups therein. Section \ref{Sec6.6}, as a main part of this paper, calculates the coHochschild homology groups, which gives an analogue of Borel-Weil-Bott theorem. As a continuation, we study the case of degree two in Section \ref{Section:degree2} and obtain a similar result. 

\subsection{Acknowledgement}
This work is part of the author's Ph.D thesis, supervised by Professor Marc Rosso. The author is grateful to his guidance and treasurable discussions on this problem.

\section{Recollections on quantum shuffle algebras}\label{Sec6.2}

We fix a field $\mc$ of characteristic $0$ in this paper. All algebras, modules, vector spaces and tensor products are over $\mc$ if not specified.

In this section, we recall the construction of quantum shuffle algebras given in \cite{Ros98}. For the basic notion of Hopf algebras, we send the reader to \cite{Swe94}.

\subsection{Symmetric groups and braid groups}\label{Section:symmbraid}
We fix some integer $n\geq 1$. Let $\mathfrak{S}_n$ denote the symmetric group acting on an alphabet with $n$ elements, say, $\{1,2,\cdots,n\}$. For an integer $0\leq k\leq n$, we let $\mathfrak{S}_{k,n-k}$ denote the set of $(k,n-k)$-shuffles in $\mathfrak{S}_n$ defined by
$$\mathfrak{S}_{k,n-k}=\{\om\in\mathfrak{S}_n|\ \om^{-1}(1)<\cdots<\om^{-1}(k),\ \om^{-1}(k+1)<\cdots<\om^{-1}(n)\}.$$
Moreover, once $\mathfrak{S}_k\times\mathfrak{S}_{n-k}$ is viewed as a subgroup of $\mathfrak{S}_n$, the multiplication gives a bijection 
\[
\xymatrix{
(\mathfrak{S}_k\times\mathfrak{S}_{n-k})\times \mathfrak{S}_{k,n-k}\ar[r]^-\sim & \mathfrak{S}_n
}
\]
which induces a decomposition of $\mathfrak{S}_n$.
\par
The braid group $\mathfrak{B}_n$ is generated by $n-1$ generators $\s_1,\cdots,\s_{n-1}$ and relations:
$$\s_i\s_j=\s_j\s_i,\ \text{if} \ |i-j|\geq 2,\ \ \s_i\s_{i+1}\s_i=\s_{i+1}\s_i\s_{i+1}\ \ \text{for}\ i=1,\cdots,n-2.$$
\par
For $1\leq i\leq n-1$, we let $s_i$ denote the transposition $(i,i+1)$ in $\mathfrak{S}_n$. Since $\mathfrak{S}_n$ can be viewed as imposing relations $\s_i^2=1$ in $\mathfrak{B}_n$, there exists a canonical projection $\pi_n:\mathfrak{B}_n\ra\mathfrak{S}_n$ by sending $\s_i$ to $s_i$.
\par
This projection admits a section in the level of set: it is a map $T: \mathfrak{S}_n\ra \mathfrak{B}_n$ sending a reduced expression $\om=s_{i_1}\cdots s_{i_k}\in\mathfrak{S}_n$ to $T_\om=\s_{i_1}\cdots \s_{i_k}\in\mathfrak{B}_n$. This map $T$ is called the Matsumoto section.

\subsection{Hopf bimodules and coinvariants}\label{Hopfbimod}
Let $H$ be a Hopf algebra with invertible antipode and $M$ be a vector space. 
\begin{definition}
$M$ is called an $H$-Hopf bimodule if it satisfies:
\begin{enumerate}
\item $M$ is an $H$-bimodule;
\item $M$ is an $H$-bicomodule with structure maps $\delta_L:M\ra H\ts M$ and $\delta_R:M\ra M\ts H$;
\item These two structures are compatible: the maps $\delta_L$ and $\delta_R$ are bimodule morphisms, where the bimodule structures on $H\ts M$ and $M\ts H$ arise from the tensor product.
\end{enumerate}
\end{definition}

If $M$ is an $H$-Hopf bimodule, it is at the same time a left $H$-Hopf module and a right $H$-Hopf module (we send the reader to \cite{Swe94} for Hopf modules).

\par
One of the most important structures for Hopf modules is the set of coinvariants as it gives a parametrization of blocks in such modules. The set of left coinvariants $M^{coL}$ and right coinvariants $M^{coR}$ are defined by:
$$M^{coL}=\{m\in M|\ \delta_L(m)=1\ts m\},\ \ M^{coR}=\{m\in M|\ \delta_R(m)=m\ts 1\}.$$

\begin{proposition}[\cite{Swe94}]\label{triviality}
Let $M$ be a right $H$-Hopf module. Then there exists an isomorphism of right $H$-Hopf modules: $M\cong M^{coR}\ts H,$
where the right hand side admits the trivial right Hopf module structure. Moreover, maps in two directions are given by:
$$M\ra M^{coR}\ts H,\ \ m\mapsto\sum P(m_{(0)})\ts m_{(1)},\ \ M^{coR}\ts H\ra M,\ \  m\ts h\mapsto mh,$$
where $m\in M$, $h\in H$, $\delta_R(m)=\sum m_{(0)}\ts m_{(1)}$ and $P:M\ra M^{coR}$ is defined by: $$P(m)=\sum m_{(0)}S(m_{(1)}).$$
\end{proposition}

We have an analogous result for left $H$-Hopf modules.
\par
Now we concentrate on the set of right coinvariants $M^{coR}$: it admits some left structures.
\begin{enumerate}
\item As $\delta_L$ and $\delta_R$ are compatible, $M^{coR}$ is a left subcomodule of $M$.
\item Once we defined the left $H$-module structure on $M$ by the adjoint action, say $h.m=\sum h_{(1)}mS(h_{(2)})$ for $h\in H$ and $m\in M$, $M^{coR}$ is a left $H$-module.
\end{enumerate}
But these will not give $M^{coR}$ a left $H$-Hopf module structure as the adjoint action is not in general a left comodule morphism. This difference raises up a structure of great interest called Yetter-Drinfel'd module.

\subsection{Yetter-Drinfel'd modules}\label{YD}
Let $H$ be a Hopf algebra. A vector space $V$ is called a (left) $H$-Yetter-Drinfel'd module if it is simultaneously an $H$-module and an $H$-comodule satisfying the Yetter-Drinfel'd compatibility condition: for any
$h\in H$ and $v\in V$,
$$\sum h_{(1)}v_{(-1)}\ts h_{(2)}.v_{(0)}=\sum (h_{(1)}.v)_{(-1)}h_{(2)}\ts (h_{(1)}.v)_{(0)},$$
where $\Delta(h)=\sum h_{(1)}\ts h_{(2)}$ and $\rho(v)=\sum v_{(-1)}\ts v_{(0)}$ are Sweedler
notations for coproducts and comodule structure maps.
\par
Morphisms between two $H$-Yetter-Drinfel'd modules are linear maps preserving $H$-module
and $H$-comodule structures.
\par
The compatibility condition between left module and comodule structures on $M^{coR}$ in the last subsection can be explained in the framework of Yetter-Drinfel'd module.

\begin{proposition}[\cite{Ros98}]\label{HopfBimodYD}
Let $H$ be a Hopf algebra. There exists an equivalence of category between the category of $H$-Hopf bimodules and the category of $H$-Yetter-Drinfel'd modules which sends a Hopf bimodule $M$ to the set of its right coinvariants $M^{coR}$.
\end{proposition}

\subsection{Quantum double construction}\label{double}
Let $A$ and $B$ be two Hopf algebras with invertible antipodes. A
generalized Hopf pairing between $A$ and $B$ is a bilinear form
$\vp:A\times B\ra\mathbb{K}$ satisfying:
\begin{enumerate}
\item For any $a\in A$, $b,b'\in B$,
$\vp(a,bb')=\sum\vp(a_{(1)},b)\vp(a_{(2)},b')$;
\item For any $a,a'\in A$, $b\in B$,
$\vp(aa',b)=\sum\vp(a,b_{(2)})\vp(a',b_{(1)})$;
\item For any $a\in A$, $b\in B$,
$\vp(a,1)=\ve(a),\ \ \vp(1,b)=\ve(b)$.
\end{enumerate}
If $\vp$ is a generalized Hopf pairing between $A$ and $B$, the quantum double $\mathcal{D}_\vp(A,B)$ is defined by:
\begin{enumerate}
\item As a vector space, it is $A\ts B$;
\item As a coalgebra, it is the tensor product of coalgebras $A$ and $B$;
\item As an algebra, the multiplication is given by:
$$(a\ts b)(a'\ts b')=\sum \vp(S^{-1}(a_{(1)}'),b_{(1)})\vp(a_{(3)}',b_{(3)})aa_{(2)}'\ts b_{(2)}b'.$$
\end{enumerate}
\par
If $H$ is a finite dimensional Hopf algebra, it is well-known that there exists an equivalence between the category of $H$-Yetter-Drinfel'd modules and the category of modules over the quantum double $\mathcal{D}_\vp(H)=\mathcal{D}_\vp(H,H^*)$ where the generalized Hopf pairing is given by the duality between $H$ and $H^*$. The following result is a consequence of Proposition \ref{HopfBimodYD}.

\begin{corollary}
Let $H$ be a finite dimensional Hopf algebra. There exists an equivalence of category between the category of $H$-Hopf bimodules and the category of modules over $\mathcal{D}_\vp(H)$ which sends $M$ to the set of its right coinvariants $M^{coR}$.
\end{corollary}

\begin{remark}\label{Rmk:double}
It should be pointed out that the corollary above holds if $H$ is a graded Hopf algebra with finite dimensional components and $H^*$ is its graded dual.
\end{remark}

\subsection{Tensor product structures}

In this subsection, we consider tensor product structures on the two categories mentioned above.
\par
Let $M$ and $N$ be two Hopf bimodules. We define an $H$-bimodule and an $H$-bicomodule structure on $M\ts_H N$ as follows:
\begin{enumerate}
\item The bimodule structure comes from the left module structure on $M$ and right module structure on $N$;
\item The bicomodule structure arises from the tensor product.
\end{enumerate}

\begin{lemma}\label{Lem:tensor}
The module and comodule structures above are well-defined. With these structures, $M\ts_H N$ is an $H$-Hopf bimodule.
\end{lemma}
\begin{proof}
The only problem is about the well-definedness of the comodule structures. We consider the case of left comodule and the right one can be shown similarly: it is clear that there is a linear map $\wt{\delta_L}:M\ts N\ra H\ts M\ts_H N$; it suffices to show that it passes through the quotient to give $\delta_L:M\ts_H N\ra H\ts M\ts_H N$. This last point can be obtained by a simple verification.
\end{proof}

The following proposition implies that the equivalence of category in the last two subsections preserves tensor product structures.

\begin{proposition}[\cite{Ros98}]
Let $M$ and $N$ be two Hopf bimodules. Then as $H$-Yetter-Drinfel'd modules, we have:
$$(M\ts_H N)^{coR}\cong M^{coR}\ts N^{coR}.$$
If moreover $H$ is of finite dimensional, the isomorphism above preserves $\mathcal{D}_\vp(H)$-module structures.
\end{proposition}

\subsection{Constructions of braiding}\label{braiding}

In the category of $H$-Hopf bimodules, Woronowicz introduced a braiding structure which is explained in \cite{Ros98}. In this subsection, we discuss the relation between braidings appearing in these three categories.

\begin{enumerate}
\item Let $M$ and $N$ be two $H$-Hopf bimodules. Then there exists a unique $H$-Hopf bimodule isomorphism $\s:M\ts_H N\ra N\ts_H M$ such that for any $\om\in M^{coL}$ and $\eta\in M^{coR}$, $\s(\om\ts\eta)=\eta\ts\om$. Moreover, $\s$ satisfies the braid equation. So the category of $H$-Hopf modules is a braided tensor category.
\item Let $V$ and $W$ be two $H$-Yetter-Drinfel'd modules. We define an isomorphism of Yetter-Drinfel'd modules $\s:V\ts W\ra W\ts V$ by 
$$\sigma(v\ts w)=\sum v_{(-1)}.w\ts v_{(0)}.$$ 
Then $\s$ satisfies the braid equation and so the category of $H$-Yetter-Drinfel'd modules is a braided tensor category.
\item If $H$ is a finite dimensional Hopf algebra, the quantum double $\mathcal{D}_\vp(H)$ is a quasi-triangular Hopf algebra and so the category of modules over $\mathcal{D}_\vp(H)$ admits a braided tensor structure where the braiding is given by the action of the R-matrix in $\mathcal{D}_\vp(H)$.
\end{enumerate}

\begin{theorem}[\cite{Ros98}]\label{equiv}
The functor sending $M$ to $M^{coR}$ is an equivalence of braided tensor category between the category of $H$-Hopf bimodules and of $H$-Yetter-Drinfel'd modules. If moreover $H$ is of finite dimensional, the above two categories are equivalent to the braided tensor category formed by $\mathcal{D}_\vp(H)$-modules.
\end{theorem}

\subsection{Tensor algebra and its dual}

Let $H$ be a Hopf algebra and $M$ be an $H$-Hopf bimodule. We have constructed a braiding on the set of its right coinvariants $M^{coR}$, which, from now on, will be denoted by $V$ for short.
\par
This construction gives a representation of braid group $\mathfrak{B}_n$ on $V^{\ts n}$ by sending $\s_i$ to $\id^{\ts (i-1)}\ts \s\ts \id^{(n-i-1)}$.
\par
We consider the tensor space
$$T(V)=\bigoplus_{n=0}^\infty V^{\ts n}$$
associated to $V$ and write $(v_1,\cdots,v_n)$ for the pure tensor $v_1\ts\cdots\ts v_n$ where $v_1,\cdots,v_n\in V$. It is well known that there is a braided Hopf algebra structure (for a definition, see \cite{AS02}) on $T(V)$ defined by:
\begin{enumerate}
\item The algebra structure is given by the concatenation;
\item The coalgebra structure is graded and is determined by: for $v_1,\cdots,v_n\in V$, the $V^{\ts p}\ts V^{\ts(n-p)}$-component of $\Delta((v_1,\cdots,v_n))$ is given by the shuffle action
$$\sum_{\s\in\mathfrak{S}_{p,n-p}}T_\s((v_1,\cdots,v_n)),$$
where $T_\s$ is the image of $\s$ under the Matsumoto section.
\end{enumerate}

If the graded dual of $T(V)$ is under consideration, we have a dual algebra and a dual coalgebra structure on it:
\begin{enumerate}
\item The algebra structure is graded and is defined by: for $v_1,\cdots,v_n\in V$, 
$$(v_1,\cdots,v_p)\ast (v_{p+1},\cdots,v_n)=\sum_{\s\in\mathfrak{S}_{p,n-p}}T_\s((v_1,\cdots,v_n));$$
\item The coalgebra structure is given by the deconcatenation:
$$\Delta((v_1,\cdots,v_n))=(v_1,\cdots,v_n)\ts 1+\sum_{p=1}^{n-1} (v_1,\cdots,v_p)\ts (v_{p+1},\cdots,v_n)+1\ts (v_1,\cdots,v_n).$$
\end{enumerate}
As shown in Proposition 9 of \cite{Ros98}, $T(V)$, with structures defined above, is a braided Hopf algebra. We let $T_\s(V)$ denote it.

\subsection{Cotensor product}

The cotensor product over a coalgebra $C$ is a dual version of the tensor product over some fixed algebra $A$. We recall the definition of cotensor product in this subsection, more information can be found in \cite{Doi81}, \cite{Nic78} and \cite{Ros98}.
\par
Let $C$ be a coalgebra and $M,N$ be two $C$-bicomodules. The cotensor product of $M$ and $N$ is a $C$-bicomodule defined as follows: we consider two linear maps $\delta_R\ts {\id}_N$, ${\id}_M\ts \delta_L:\ M\ts N\ra M\ts C\ts N$; the cotensor product of $M$ and $N$, denoted by $M\Box_C N$, is the equalizer of $\delta_R\ts {\id}_N$ and ${\id}_M\ts \delta_L$.

\subsection{Quantum shuffle algebras and their bosonizations}\label{Section:QSA}
In this subsection, we recall the definition of the quantum shuffle algebra given in \cite{Ros98} and \cite{Ros02}. Notations in previous subsections will be adopted.
\par
We start from considering the linear map $V\ra T_\s(V)$ given by the identity map. From the universal property of $T(V)$ as an algebra, we obtain a graded algebra morphism $\pi: T(V)\ra T_\s(V)$.
\begin{definition}
The image of the graded algebra morphism $\pi$ is called the quantum shuffle algebra and will be denoted by $S_\s(V)$; it is a subalgebra of $T_\s(V)$.
\end{definition}
For any integer $n\geq 1$, we define two elements
$$S_n=\sum_{\s\in\mathfrak{S}_n}\s\in \mc[\mathfrak{S}_n],\ \ \Sigma_n=\sum_{\s\in\mathfrak{S}_n}T_\s\in \mc[\mathfrak{B}_n].$$
The element $\Sigma_n$, once acting on $V^{\ts n}$, is called a total symmetrization operator. The following lemma hides between lines of \cite{Ros98}.
\begin{lemma}[\cite{Ros98}]
When restricted to $V^{\ts n}$, the map $\pi:T(V)\ra T_\s(V)$ is given by the total symmetrization operator $\Sigma_n$.
\end{lemma}
This gives the following isomorphism of braided Hopf algebras
$$\overline{\pi}: T(V)\left/\bigoplus_{n=2}^\infty\ker\Sigma_n\right.\cong S_\s(V).$$

\begin{remark}
Up to the total symmetrization map, the Nichols algebra defined in \cite{Nic78} and \cite{AS02} is isomorphic to the quantum symmetric algebra as braided Hopf algebra.
\end{remark}

\par
At last, we describe the bosonization of quantum shuffle algebras, following \cite{Ros98}.
\par
In fact, instead of considering only the set of right coinvariants, we can start with the $H$-Hopf bimodule $M$ and consider the cotensor Hopf algebra 
$$T^\Box_H(M)=H\oplus \left(\bigoplus_{n=1}^\infty M^{\Box_H n}\right).$$
We let $S_H(M)$ denote the sub-Hopf algebra of $T^\Box_H(M)$ generated by $H$ and $M$. It is an $H$-Hopf bimodule and the set of its right coinvariants is isomorphic to $S_\s(V)$ as an algebra; moreover, as an algebra, $S_H(M)$ is isomorphic to the crossed product of $H$ and $S_\s(V)$.
\par
This $S_H(M)$ is called the bosonization of $S_\s(V)$ by the Hopf algebra $H$.

\subsection{Construction of quantum groups}\label{ConstructionOfQG}
In this subsection, we recall the construction of the strictly negative part of a quantized enveloping algebra (quantum group) as a quantum shuffle algebra \cite{Ros98}.
\par
Let $H=\mc[G]$ be the group algebra of a finitely generated abelian group $G$. According to the classification of finite rank $\mathbb{Z}$-modules, there exists an integer $r\geq 0$ and some positive integers $l_1, \cdots, l_s$ such that 
$$G\cong \mathbb{Z}^r\times \left(\prod_{i=1}^s\mathbb{Z}/l_i\right).$$
We suppose that $n=r+s$ is the rank of $G$. From the argument before Lemma 14 in \cite{Ros98}, if $V=M^{coR}$ is the set of right coinvariants of an $H$-Hopf bimodule $M$ and $(F_1,\cdots,F_n)$ is a basis of the vector space $V$, then the braiding on $V$ can be characterized by a square matrix of $n^2$ numbers $q_{ij}$, $1\leq i,j\leq n$, which is called the braiding matrix.
\par
To be more concrete, if we let $K_1,\cdots,K_n$ denote a free $\mathbb{Z}$-basis of $G$, then the left $H$-comodule structure on $V$ is given by $\delta_L(F_i)=K_i^{-1}\ts F_i$ and the left $H$-module structure is determined by $K_i.F_j=q_{ij}^{-1}F_j$. With this construction, the braiding is characterized by 
$$\s(F_i\ts F_j)=K_i^{-1}.F_j\ts F_i=q_{ij}F_j\ts F_i.$$
In particular, if the braiding matrix comes from data in Lie theory, for example, a symmetrizable generalized Cartan matrix, the quantum shuffle algebra constructed above is of great interest.
\par
Let $C=(c_{ij})\in M_n(\mathbb{Z})$ be a symmetrizable generalized Cartan matrix and $(d_1,\cdots,d_n)$ be positive integers such that the matrix $A=(d_ic_{ij})=(a_{ij})$ is symmetric. We choose $q\neq 0,\pm 1$ be an element in $k$ and define the braiding matrix $(q_{ij})\in M_n(\mathbb{C})$ by $q_{ij}=q^{a_{ij}}$.
\par
For each symmetrizable generalized Cartan matrix $C$, we can associate to it a Kac-Moody Lie algebra $\g(C)$. After Drinfel'd and Jimbo, there exists a corresponding quantized enveloping algebra $U_q(\g(C))$ defined by generators and relations. The following theorem permits us to give a functorial construction of the (strictly) negative part of such algebras.

\begin{theorem}[\cite{Ros98}]\label{RosMain}
After the construction above, we have:
\begin{enumerate}
\item Let $G=\mathbb{Z}^n$ and $q$ not be a root of unity. Then the quantum shuffle algebra $S_\s(V)$ is isomorphic, as a braided Hopf algebra, to the strictly negative part of $U_q(\g(C))$. Moreover, the bosonization $S_H(M)$ is isomorphic, as a Hopf algebra, to the negative part of $U_q(\g(C))$.
\item We fix a positive integer $l\geq 3$ and a primitive $l$-th root of unity $q$. Let $G=(\mathbb{Z}/l)^n$. Then the quantum shuffle algebra $S_\s(V)$ is isomorphic to the strictly negative part of the restricted quantized enveloping algebra $u_q(\g(C))$. Moreover, the bosonization $S_H(M)$ is isomorphic, as a Hopf algebra, to the quotient of the negative part of $u_q(\g(C))$ by the Hopf ideal generated by $K_i^l-1$, $i=1,\cdots,n$.
\end{enumerate}
\end{theorem}

To obtain the whole quantum group, it suffices to double the bosonization $S_H(M)$ using the quantum double construction given in Section \ref{double} then identify two copies of $H$. The book \cite{KRT94} can be served as a good reference for this construction.
\par
The isomorphism given in Section \ref{Section:QSA}
$$\overline{\pi}: T(V)\left/\bigoplus_{n=2}^\infty\ker\Sigma_n\right.\cong S_\s(V),$$
implies that the kernel $\bigoplus_{n=2}^\infty \ker\Sigma_n$ is indeed generated by quantized Serre relations in $U_q(\g)$.

\section{Main Construction and Rosso's theorem}\label{Section:mainconstruction}

\subsection{Data}\label{Sec6.4.1}

We preserve assumptions in Section \ref{ConstructionOfQG} and fix the following notations:
\begin{enumerate}
\item $(\mathfrak{h},\Pi,\Pi^\vee)$ is a realization of the generalized Cartan matrix $C$ and $W$ is the Weyl group.
\item $(\cdot,\cdot)$ is a $W$-invariant bilinear form on $\mathbb{Q}\Pi$ such that $(\alpha_i,\alpha_j)=a_{ij}$.
\item $\mathcal{P}=\{\lambda\in\mathfrak{h}^*|\ (\lambda,\alpha_i)\in\mathbb{Z}, \forall i=1,\cdots,n\}$ is that weight lattice and $\mathcal{P}_+\subset\mathcal{P}$ is the set of dominant weights in $\g$.
\item For $\lambda\in\mathcal{P}$, $L(\lambda)$ is the unique (up to isomorphism) irreducible representation of $\g$ of highest weight $\lambda$.
\end{enumerate}

\par
If $q$ is a primitive root of unity such that $q^l=1$ for some odd $l\geq 3$, we let $\mathcal{P}^l$ denote the set $\{\lambda\in\mathcal{P}|\, |(\lambda,\alpha_i)|<l\}$ and $\mathcal{P}_{+}^l=\mathcal{P}^l\cap\mathcal{P}_{+}$.
\par
For a weight $\lambda\in\mathcal{P}$ or $\mathcal{P}^l$, we let $K$ denote the commutative Hopf algebra generated by $H$ together with group-like elements $K_\lambda^{\pm 1}$ such that $K_\lambda K_\lambda^{-1}=K_{\lambda}^{-1}K_\lambda=1$ and let $W$ denote the vector space generated by $V$ and a vector $v_\lambda$. We dispose the following structures on $W$. 

\begin{enumerate}
\item $W$ is a left $K$-comodule with the structural map give by: when restricted to $V$, $\delta_L$ is the $H$-comodule structural map of $V$, $\delta_L(v_\lambda)=K_\lambda^{-1}\ts v_\lambda$;
\item $W$ is a left $K$-module with the module structure given by: 
$$K_\lambda.F_i= q^{(\lambda,\alpha_i)} F_i,\,\, K_\lambda.v_\lambda=q^{-2}v_\lambda \,\, \text{and}\,\, K_i.v_\lambda=q^{(\lambda,\alpha_i)}v_\lambda.$$ 
The other actions coincide with those in $V$.
\end{enumerate}
As a vector space, we set $N=W\ts K$; it admits an $K$-Hopf bimodule structure after the following definitions:
\begin{enumerate}
\item Right module and comodule structures are trivial: i.e., they come from the regular right $K$-module and $K$-comodule;
\item Left module and comodule structures come from the tensor product.
\end{enumerate}
\par
Then it is clear that when this structure is under consideration, $N^{coR}=W$.
\par
Starting with this $K$-Hopf bimodule $N$, we can construct the corresponding quantum shuffle algebra $S_{\widetilde{\sigma}}(W)$ and its bosonization $S_{K}(N)$. An easy computation gives the following formula for the braiding $\widetilde{\s}$ on $W\ts W$: when restricted to $V\ts V$, it coincides with $\s$;
$$\wt{\sigma}(F_i\ts v_\lambda)=q^{-(\lambda,\alpha_i)}v_\lambda\ts F_i,\ \ \wt{\sigma}(v_\lambda\ts v_\lambda)=q^2v_\lambda\ts v_\lambda\ \ \widetilde{\s}(v_\lambda\ts F_i)=q^{-(\lambda,\alpha_i)}F_i\ts v_\lambda.$$
\par
We dispose a gradation on $S_{\widetilde{\sigma}}(W)$ by letting $\deg(F_i)=0$ and $\deg(v_\lambda)=1$. We let $S_{\widetilde{\sigma}}(W)_{(k)}$ denote the subspace of $S_{\widetilde{\sigma}}(W)$ containing elements of degree $k$. For any $k>0$, $S_{\wt{\s}}(W)_{(k)}$ does not admit an algebra structure.
\par
Since $V\subset W$ is a $K$-submodule and $H\subset K$ is a sub-Hopf algebra, $V$ is a sub-$K$-Yetter-Drinfel'd module of $W$. It is then clear that $(V,\s)$ is a sub-braided vector space of $(W,\wt{\s})$ and $S_\s(V)$ is a sub-braided Hopf algebra of $S_{\wt{\s}}(W)$ in the category of $K$-Yetter-Drinfel'd modules.

\subsection{Hopf bimodule structure on $S_{\wt{\sigma}}(W)$}\label{HopfBim}

With the gradation defined in the end of last subsection, $S_{\wt{\s}}(W)$ is a graded braided Hopf algebra with $S_{\wt{\s}}(W)_{(0)}=S_\s(V)$.
\par
The projection $p:S_{\wt{\s}}(W)\ra S_\s(V)$ onto degree $0$ and the embedding $i:S_\s(V)\ra S_{\wt{\s}}(W)$ into degree $0$ are both braided Hopf algebra morphisms. This gives $S_{\wt{\s}}(W)$ a braided-$S_\s(V)$-Hopf bimodule structure: left and right comodule structural maps are $(p\ts\id)\Delta$ and $(\id\ts p)\Delta$; left and right module structure are induced by the inclusion $i:S_\s(V)\ra S_{\wt{\s}}(W)$.
\par
As elements in $S_\s(V)$ are of degree $0$, for each $k\in\mathbb{N}$, $S_{\wt{\s}}(W)_{(k)}$ inherits an $S_\s(V)$-braided Hopf bimodule structure. To simplify notations, we let $M_k$ denote $S_{\wt{\s}}(W)_{(k)}$ and $M$ denote $S_{\wt{\s}}(W)$.

\subsection{A theorem of Rosso}\label{Sec:Rosso}
In this subsection, we explain a theorem due to M. Rosso which computes the coinvariant space $M_1^{coR}$.
\par
As $M$ and each $M_k$ for $k\in\mathbb{N}$ are right braided $S_\s(V)$-Hopf modules, we let $M^{coR}$ and $M_k^{coR}$ denote the set of their right coinvariants, respectively. The braided version of the structural theorem of Hopf modules can be then applied to give the following isomorphisms of right braided Hopf modules:
$$S_{\widetilde{\sigma}}(W)\cong  M^{coR}\ts S_\s(V),\ \ S_{\wt{\sigma}}(W)_{(k)}\cong M_k^{coR}\ts S_\s(V).$$
\par
We discuss the module and comodule structures on these sets of coinvariants:
\begin{enumerate}
\item As we have explained in Section \ref{Hopfbimod}, for any $k\in\mathbb{N}$, $M^{coR}$ and $M_k^{coR}$ admit adjoint $S_\sigma(V)$-module structures.
\item $M^{coR}$ and $M_k^{coR}$ are all $S_\sigma(V)$-left comodules: they are induced by the left comodule structures on $M$ and $M_k$, respectively.
\item The module and comodule structures on $M^{coR}$ and $M_k^{coR}$ are compatible in the sense of Yetter-Drinfel'd. Thus both of them are $S_\s(V)$-Yetter-Drinfel'd modules.
\end{enumerate}

In fact, we can use the bosonization procedure to avoid all prefixes "braided". As both $S_\s(V)$ and $S_{\wt{\s}}(W)$ are in the category of $K$-Yetter-Drinfel'd modules, the bosonization with $K$ gives two Hopf algebras $S_K(M)$ and $S_K(N)$. If the Hopf algebra $K$ is designated to be of degree $0$, the projection onto degree $0$ and the embedding into degree $0$ endow $S_K(N)$ and all $S_K(N)_{(k)}$ Hopf bimodule structures over $S_K(M)$. The structural theorem of right Hopf modules can be then applied to give 
$$S_K(N)\cong M^{coR}\ts S_K(M),\ \ S_K(N)_{(k)}\cong M^{coR}_k\ts S_K(M).$$
Then $M^{coR}$ and $M^{coR}_k$ are in the category of $S_K(M)$-Yetter-Drinfel'd modules. According to Theorem \ref{equiv}, they admit $\mathcal{D}_\vp(S_K(M))$-module structures. Moreover, as $\mathcal{D}_\vp(S_H(M))$ is a sub-Hopf algebra of $\mathcal{D}_\vp(S_K(M))$, they admit $\mathcal{D}_\vp(S_H(M))$-structures. 
\par
If the generalized Hopf pairing is carefully chosen (for example, we take the pairing in the definition of the quantum group as a double), the $H$-action and the dual of $H$-coaction with respect to the pairing coincide. As a consequence of Theorem \ref{RosMain}, $M^{coR}$ and $M_k^{coR}$ for any $k\in\mathbb{N}$ admit $U_q(\g)$-module structures.
\par
The following theorem is due to Rosso \cite{Ros10}.

\begin{theorem}\label{RossoThm}
$M_1^{coR}$ is an irreducible $U_q(\g)$-module of highest weight $\lambda$, so it is isomorphic to $L(\lambda)$.
\end{theorem}

\section{Coalgebra homology and module structures}\label{Sec6.5}

\subsection{Hochschild homology of an algebra}
Let $A$ be an associative algebra and $M$ be an $A$-bimodule.
\par
The Hochschild homology of $A$ with coefficients in $M$ is defined as the homology of the complex $(C_\bullet(A,M),d)$, where
$$C_n(A,M)=M\ts A^{\ts n},$$
and the differential $d:C_n(A,M)\ra C_{n-1}(A,M)$ is given by: for $a_1,\cdots,a_n\in A$ and $m\in M$,
\begin{eqnarray*}
d(m\ts a_1\ts\cdots\ts a_n) &=& ma_1\ts a_2\ts\cdots\ts a_n\\
& &\, +\sum_{i=1}^{n-1}(-1)^i m\ts a_1\ts\cdots\ts a_ia_{i+1}\ts\cdots\ts a_n\\
& &\,\, +(-1)^n a_nm\ts a_1\ts\cdots\ts a_{n-1}.
\end{eqnarray*}
We denote ${\HH}_n(A,M)={\HH}_n(C_\bullet(A,M),d)$ the $n$-th homology group of the complex $(C_\bullet(A,M),d)$.

\subsection{coHochschild homology of a coalgebra}
The coHochschild homology of a coalgebra, firstly studied by P. Cartier, is a dual construction of the Hochschild homology of an algebra.

\par
Let $C$ be a coalgebra and $N$ be a $C$-bicomodule. We recall the coHochschild homology of $C$ with coefficients in $N$ given in \cite{Doi81}.
\par
We let $R^n(C,N)=N\ts C^{\ts n}$ and define the differential $\delta:N\ts C^{\ts n}\ra N\ts C^{\ts (n+1)}$ by:
\begin{eqnarray*}
\delta(n\ts c_1\ts\cdots\ts c_n)&=&\delta_R(n)\ts c_1\ts\cdots\ts c_n\\
& &\ +\sum_{i=1}^n (-1)^i n\ts c_1\ts\cdots\ts \Delta(c_i)\ts\cdots\ts c_n\\
& &\ \ +(-1)^{n+1}\sum n_{(0)}\ts c_1\ts\cdots\ts c_n\ts n_{(-1)},
\end{eqnarray*}
where $\delta_L$, $\delta_R$ are $C$-bicomodule structural maps and for $n\in N$, $\delta_L(n)=\sum n_{(-1)}\ts n_{(0)}$. The coHochschild homology of $C$ with coefficients in $N$ is then defined by 
$${\Hoch}^i(C,N)={\HH}^i(R^\bullet(C,N),\delta),$$
where $\HH^i(R^\bullet(C,N),\delta)$ is the $i$-th cohomology group of the complex $(R^\bullet(C,N),\delta)$.

\subsection{Module and comodule structures on coHochschild homology groups}\label{Section:modcomod}

In this subsection, we discuss how module and comodule structures on coefficients induce such structures on the coHochschild homology groups.\\
\indent
At first, we suppose that $C$ and $D$ are two Hopf algebras. Comodule structures which we will work with are defined by:
\begin{enumerate}
\item $N$ is a $C$-bicomodule such that the left $C$-comodule structure is trivial.
\item $N$ admits a $(D,C)$-comodule structure; that is to say, $N$ admits a left $D$-comodule structure compatible with its right $C$-comodule structure. 
\item The coalgebra $C$ admits a trivial left $D$-comodule structure.
\end{enumerate}

\par
Then for any integer $i\geq 0$, $N\ts C^{\ts i}$ admits a left $D$-comodule structure given by the tensor product when structures above are under consideration:
$$\delta_L^D(n\ts c^1\ts\cdots\ts c^i)=\delta_L^D(n)\ts c^1\ts\cdots\ts c^i\in D\ts N\ts C^{\ts i}.$$

\begin{proposition}\label{Prop:comodule}
For any integer $i\geq 0$, $\Hoch^i(C,N)$ inherits a $D$-comodule structure.
\end{proposition}
\begin{proof}[Proof]
We show that for any $i\geq 0$, 
$$d^i:N\ts C^{\ts i}\ra N\ts C^{\ts (i+1)}$$
is a $D$-comodule morphism. It suffices to show the commutativity of the following diagram:
\[
\xymatrix{
N\ts C^{\ts i} \ar[r]^-{d^i} \ar[d]^-{\delta_L^D} & N\ts C^{\ts (i+1)} \ar[d]^-{\delta_L^D} \\
D\ts N\ts C^{\ts i} \ar[r]^-{{\id}_D\ts d^i} & D\ts N\ts C^{\ts (i+1)}.
}
\]
We take an element $n\ts c^1\ts\cdots\ts c^i\in N\ts C^{\ts i}$, then
\begin{eqnarray*}
& & ({\id}_D\ts d^i)(\delta_L^D(n\ts c^1\ts\cdots\ts c^i))\\
&=& \sum n_{(-1)}\ts n_{(0)}\ts n_{(1)}\ts c^1\ts\cdots\ts c^i\\
& & \ \ +\sum_{p=1}^i (-1)^p\sum n_{(-1)}\ts n_{(0)}\ts c^1\ts\cdots\ts c^p_{(1)}\ts c^p_{(2)}\ts\cdots\ts c^i\\
& &\ \ \ \ +(-1)^{i+1}\sum n_{(-1)}\ts n_{(0)}\ts c^1\ts\cdots\ts c^i\ts 1\\
&=& \delta_L^D(d^i(n\ts c^1\ts\cdots\ts c^i)).
\end{eqnarray*}
As a consequence, $\delta_L^D$ induces $\delta_L^D: \Hoch^i(C,N)\ra D\ts \Hoch^i(C,N)$, which gives a $D$-comodule structure on $\Hoch^i(C,N)$.
\end{proof}

Moreover, we consider the module structure on the coHochschild homology groups.
\par
At this time, we suppose that following data are given:
\begin{enumerate}
\item $C$ is a trivial $D$-bimodule given by the counit $\ve$;
\item $M$ is a $C$-bicomodule where the left $C$-comodule is trivial;
\item $M$ is a $D$-bimodule, then it is a left adjoint $D$-module;
\item the right $C$-comodule structural map on $M$ is a $D$-bimodule morphism.
\end{enumerate}
Then for any integer $i\geq 0$, $M\ts C^{\ts i}$ admits a $D$-module structure arising from the tensor product:
$$d.(m\ts c^1\ts\cdots\ts c^i)=\sum d_{(1)}mS(d_{(2)})\ts c^1\ts\cdots\ts c^i.$$

\begin{proposition}\label{Prop:module}
For any integer $i\geq 0$, $\Hoch^i(C,M)$ inherits a $D$-module structure.
\end{proposition}
\begin{proof}[Proof]
As in the last proposition, it suffices to show that for any $i\geq 0$, $d^i:M\ts C^{\ts i}\ra M\ts C^{\ts (i+1)}$ is a morphism of $D$-module, that is to say, the following diagram commutes:
\[
\xymatrix{
D\ts M\ts C^{\ts i} \ar[r]^-{{\id}_D\ts d^i} \ar[d]^-{a} & D\ts M\ts C^{\ts (i+1)} \ar[d]^-{a} \\
M\ts C^{\ts i} \ar[r]^-{d^i} & M\ts C^{\ts (i+1)},
}
\]
where $a$ is the adjoint left $D$-module structural map.
\par
We take an element $d\ts m\ts c^1\ts\cdots\ts c^i\in D\ts M\ts C^{\ts i}$, then
\begin{eqnarray*}
& & a\circ(\id\ts d^i)(d\ts m\ts c^1\ts\cdots\ts c^i)\\
&=& \sum d_{(1)}m_{(0)}S(d_{(2)})\ts m_{(1)}\ts c^1\ts\cdots\ts c^i\\
& &\ \ +\sum_{p=1}^i (-1)^p d_{(1)}mS(d_{(2)})\ts c^1\ts\cdots\ts c_{(1)}^p\ts c_{(2)}^p\ts\cdots\ts c^i\\
& &\ \ \ \  +(-1)^{i+1}\sum d_{(1)}mS(d_{(2)})\ts c^1\ts\cdots\ts c^i\ts 1.
\end{eqnarray*}
On the other side, 
\begin{eqnarray*}
& & d^i\circ a(d\ts m\ts c^1\ts\cdots\ts c^i)\\
&=& d^i\left(\sum d_{(1)}mS(d_{(2)})\ts c^1\ts\cdots\ts c^i\right)\\
&=& a\circ(\id\ts d^i)(d\ts m\ts c^1\ts\cdots\ts c^i).
\end{eqnarray*}
\end{proof}

As a summary, suppose that there exist two Hopf algebras $C$, $D$ and a vector space $M$ satisfying the following conditions:
\begin{enumerate}
\item $C$ is a trivial $D$-Hopf bimodule;
\item $M$ is a $C$-bicomodule where the left comodule structure is trivial;
\item $M$ is a left $D$-comodule such that $M$ is a $(D,C)$-bicomodule;
\item $M$ is a $D$-bimodule such that the right $C$-comodule structural map is a $D$-bimodule morphism.
\end{enumerate}
Then for any $i\in\mathbb{N}$, $\Hoch^i(C,M)$ inherits a left adjoint $D$-module structure and a left $D$-comodule structure from the corresponding ones on $M$.

\section{A Borel-Weil-Bott type theorem}\label{Sec6.6}
In this section, we compute the coHochschild homology of $S_\s(V)$ with coefficient in the bicomodule $S_{\widetilde{\s}}(W)_{(1)}$ to obtain a Borel-Weil-Bott type theorem. These can be viewed as an analogue of the flag variety and an equivariant line bundle over it respectively. We explain in the following correspondence:
\begin{enumerate}
\item The quantum shuffle algebra $S_\s(V)$ can be viewed as an analogue of a non-commutative object corresponding to the flag variety $G/B$.
\item The $S_\s(V)$-Hopf bimodule $S_{\wt{\s}}(W)_{(1)}$ generated by one vector (for example, $v_\lambda$) is an analogous of the line bundle $\mathcal{L}(\lambda)$ generated equivariantly by a vector of weight $\lambda$ over $G/B$.
\item The set of coinvariants in $S_{\wt{\s}}(W)_{(1)}$ is an analogue of the set of global invariants (more precisely, the set of global sections) on $\mathcal{L}(\lambda)$.
\end{enumerate}

\subsection{Main construction}\label{Sec:Cons}
For this cohomological purpose, we need to do a little change for the $S_\s(V)$-module and comodule structures on $S_{\wt{\s}}(W)$.

\begin{enumerate}
\item The right $S_\s(V)$-module structure is given by the multiplication in $S_{\widetilde{\sigma}}(W)$.
\item The left $S_\s(V)$-comodule structure on $S_{\widetilde{\sigma}}(W)$ is defined by: 
$$\delta_L:S_{\widetilde{\sigma}}(W)\ra S_\sigma(V)\ts S_{\widetilde{\sigma}}(W),\ \  F_i\mapsto 1\ts F_i,\  v_\lambda\mapsto 1\ts v_\lambda.$$
\item  The right $S_\s(V)$-comodule structure on $S_{\widetilde{\sigma}}(W)$ is given by: 
$$\delta_R:S_{\widetilde{\sigma}}(W)\ra S_{\widetilde{\sigma}}(W)\ts S_\sigma(V),\ \  F_i\mapsto F_i\ts 1+1\ts F_i,\  v_\lambda\mapsto v_\lambda\ts 1.$$
\end{enumerate}
That is to say, we trivialize the left comodule structure and make right structures being untouched. It is clear that these structures descend to $S_{\wt{\s}}(W)_{(k)}$ for any $k\in\mathbb{N}$.

\begin{lemma}
With structures defined above, $S_{\widetilde{\sigma}}(W)$ and $S_{\widetilde{\sigma}}(W)_{(k)}$ for $k\in\mathbb{N}$ are right $S_\sigma(V)$-Hopf modules and $S_{\wt{\s}}(W)$-bicomodules.
\end{lemma}
It should be pointed out that as we do not touch the right comodule structure, the set of right coinvariants will be the same as the original case. That is to say, if we let $M_1$ denote $S_{\wt{\s}}(W)_{(1)}$ with the above module and comodule structures, then as vector space, $M_1^{coR}\cong L(\lambda)$.

\subsection{Calculation of $\Hoch^0$}
Now we proceed to compute the degree $0$ coHochschild homology group of $S_{\s}(V)$ as a coalgebra with coefficient in the $S_\s(V)$-bicomodule $S_{\wt{\s}}(W)_{(1)}$. We point out that this will only use the $S_\s(V)$-bicomodule structure on $M_1$.
\par
It is better to start with a general framework. This will be useful to explain the set of coinvariants as some "global sections".
\par
Let $C$ be a Hopf algebra and $M$ be a right $C$-comodule. We give $M$ a trivial left $C$-comodule structure by defining $\delta_L(m)=1\ts m$. Then $M$ is a $C$-bicomodule.

\begin{proposition}\label{degree0}
With assumptions above, $\Hoch^0(C,M)=M^{coR}$.
\end{proposition}
\begin{proof}[Proof]
We compute $\Hoch^0(C,M)$. By definition, this is given by the kernel of 
$$d^0:M\ra M\ts C,\ \ m\mapsto \sum m_{(0)}\ts m_{(-1)}-\sum m_{(0)}\ts m_{(1)},$$
where $\delta_L(m)=\sum m_{(-1)}\ts m_{(0)}$ and $\delta_R(m)=\sum m_{(0)}\ts m_{(1)}$.\\
\indent
The trivialization of the left comodule structure implies that $d^0(m)=0$ if and only if $\sum m_{(0)}\ts m_{(1)}=m\ts 1$, which is equivalent to $m\in M^{coR}$.
\end{proof}

According to this proposition, $\Hoch^0(S_\s(V),S_{\widetilde{\s}}(W)_{(1)})\cong L(\lambda)$ as vector space. Now we will endow them with module and comodule structures using the construction in Section \ref{Section:modcomod}.
\par
To make notations more transparent, we denote $C=D=S_\s(V)$. The $S_\s(V)$-structures given at the beginning of this section on $S_{\widetilde{\sigma}}(W)$ are treated as $C$-module and comodule structures.
\par
Now we define the $D$-module and comodule structures on $S_{\widetilde{\sigma}}(W)$:
\begin{enumerate}
\item The left $D$-comodule structure is given by: 
$$\delta_L:S_{\widetilde{\sigma}}(W)\ra D\ts S_{\widetilde{\sigma}}(W),\ \  F_i\mapsto F_i\ts 1+1\ts F_i,\  v_\lambda\mapsto 1\ts v_\lambda.$$
\item The $D$-bimodule structure comes from the multiplication and it makes $S_{\wt{\s}}(W)$ an adjoint $D$-module.
\end{enumerate}
Then for any $k\in\mathbb{N}$, $S_{\wt{\s}}(W)_{(k)}$ inherits these structures.
\par
It is clear that these $C,D$-module and comodule structures on $M$ satisfy the hypothesis in the end of the last section. According to Proposition \ref{Prop:comodule} and \ref{Prop:module}, for any $i,k\in\mathbb{N}$, the homology group $\Hoch^i(S_\s(V),S_{\wt{\s}}(W)_{(k)})$ admits a $D$-comodule structure and a $D$-module structure given by the adjoint action. These module and comodule structures satisfy the Yetter-Drinfel'd compatibility condition as explained in the end of Section \ref{Hopfbimod}.
\par
In fact, as shown in the end of Section \ref{Sec:Rosso}, taking bosonizations of $S_\s(V)$ and $S_{\wt{\s}}(W)$ with $K$ does not change the set of right coinvariants but will give the homology group $\Hoch^i(S_\s(V),S_{\wt{\s}}(W)_{(k)})$ an $S_K(M)$-Yetter-Drinfel'd module structure and therefore a $\mathcal{D}(S_K(M))$-module structure. Thus all homology groups $\Hoch^i(S_\s(V),S_{\wt{\s}}(W)_{(k)})$ for $i,k\in\mathbb{N}$ admit $U_q(\g)$-module structures.

\begin{corollary}
For a weight $\lambda\in\mathcal{P}$ (if $q^l=1$ is a primitive root of unity, $\lambda\in\mathcal{P}^l$), $\Hoch^0(S_\s(V),S_{\widetilde{\s}}(W)_{(1)})$ is isomorphic to $L(\lambda)$ as $U_q(\g)$-modules.
\end{corollary} 

The following part of this section is devoted to computing the higher coHochschild homology groups in the generic and root of unity cases respectively.

\subsection{Duality between Hochschild and coHochschild homologies}
We start with a general setting: suppose that
\[
\xymatrix{
C_\bullet: \cdots \ar[r]^-d & C_n\ar[r]^-d & \cdots\ar[r]^-d & C_2\ar[r]^-d & C_1\ar[r]^-d & M,
}
\]
\[
\xymatrix{
C'_\bullet: \cdots  & C_n\ar[l]_-\delta & \cdots\ar[l]_-\delta & C_2\ar[l]_-\delta & C_1\ar[l]_-\delta & M \ar[l]_-\delta
}
\]
are two complexes of finite dimensional vector spaces where $M$ is on degree $0$ and $C_i$ on degree $i$ such that
\begin{enumerate}
\item for each $i=1,2,\cdots$, there exists a bilinear form $\vp_i:C_i\times C_i\ra\mc$;
\item there is a bilinear form $\vp_0:M\times M\ra \mc$;
\item differentials $d$ and $\delta$ are adjoint to each other with respect to these pairings.
\end{enumerate}
We let $\HH_\bullet(M)$ (resp. $\HH^\bullet(M)$) denote the homology group of the complex $C_\bullet$ (resp. $C'_\bullet$). Since these complexes have differentials which are adjoint to each other, the bilinear forms $\vp_0,\vp_1,\cdots,\vp_n,\cdots$ induce a family of bilinear forms $\overline{\vp}_n:{\HH}_n(M)\times {\HH}^n(M)\ra \mc$ for any $n\geq 0$.
\par
If moreover these bilinear forms are non-degenerate, we have the following duality result whose proof is direct.

\begin{proposition}\label{duality}
Let $\vp_0,\vp_1,\cdots,\vp_n,\cdots$ be non-degenerate bilinear forms. Then for any $n\geq 0$, $\overline{\vp}_n:{\HH}_n(M)\times {\HH}^n(M)\ra \mc$ is non-degenerate.
\end{proposition}

\subsection{Application to quantum shuffle algebra}\label{application}
We fix a non-degenerate graded Hopf pairing on $S_{\widetilde{\s}}(W)$: the existence of such a pairing is explained in \cite{Chen07}. This subsection is devoted to proving the following result:
\begin{proposition}\label{dual}
For any $n\geq 0$, the bilinear form 
$$\vp:{\HH}_n(S_\s(V),S_{\widetilde{\s}}(W)_{(1)})\times {\Hoch}^n(S_\s(V),S_{\widetilde{\s}}(W)_{(1)})\ra \mc$$ 
is non-degenerate.
\end{proposition}

We consider the Bar complex of $S_{\widetilde{\s}}(W)_{(1)}$
\[
\xymatrix{
S_{\widetilde{\s}}(W)_{(1)}\ts S_\s(V)^{\ts n}\ar[r]^-d &  \cdots \ar[r]^-d & S_{\widetilde{\s}}(W)_{(1)}\ts S_\s(V)\ar[r]^-d & S_{\widetilde{\s}}(W)_{(1)}
}
\]
and the coBar complex
\[
\xymatrix{
S_{\widetilde{\s}}(W)_{(1)}\ar[r]^-\delta & S_{\widetilde{\s}}(W)_{(1)}\ts S_\s(V) \ar[r]^-\delta & \cdots\ar[r]^-\delta & S_{\widetilde{\s}}(W)_{(1)}\ts S_\s(V)^{\ts n}\ar[r]^-\delta &\cdots.
}
\]
The $S_\s(V)$-bimodule and bicomodule structures on $S_{\widetilde{\s}}(W)_{(1)}$ are given by:
\begin{enumerate}
\item The left $S_\s(V)$-module structure on $S_{\widetilde{\s}}(W)_{(1)}$ is given by multiplication and the right module structure is given by the augmentation map $\ve$;
\item The left $S_\s(V)$-comodule structure on $S_{\widetilde{\s}}(W)_{(1)}$ is trivial and the right comodule structure is given by the comultiplication in $S_{\widetilde{\s}}(W)_{(1)}$.
\end{enumerate}
According to the definition, $S_\s(V)$-bimodule and bicomodule structures on $S_{\widetilde{\s}}(W)_{(1)}$ are in duality and differentials $d$ and $\delta$ are adjoint to each other with respect to the pairing.

\begin{remark}
The computation of the vector space $\Hoch^i(S_\s(V),S_{\wt{\s}}(W)_{(1)})$ only concerns the bicomodule structure on $S_{\wt{\s}}(W)_{(1)}$, so we can certainly change the module structures defined in the beginning of Section \ref{Sec:Cons}.
\end{remark}

\par
Moreover, both $S_\s(V)$ and $S_{\widetilde{\s}}(W)_{(1)}$ have length gradations induced by the cotensor coalgebra: 
$$S_\s(V)=\bigoplus_{n=0}^\infty S_\s^n(V),\ \ S_{\widetilde{\s}}(W)_{(1)}=\bigoplus_{n=0}^\infty S_{\widetilde{\s}}^n(W)_{(1)},$$
where $S_\s^n(V)$ and $S_{\widetilde{\s}}^n(W)_{(1)}$ are linearly generated by monomials of length $n$.
\par
Then the Bar and coBar complexes admit gradations induced by those on $S_\s(V)$ and $S_{\widetilde{\s}}(W)_{(1)}$: for example, elements of degree $p$ in $S_{\widetilde{\s}}(W)_{(1)}\ts S_\s(V)^{\ts n}$ are formed by:
$$\bigoplus_{i_0+\cdots+i_n=p}S_{\widetilde{\s}}^{i_0}(W)_{(1)}\ts S_\s^{i_1}(V)\ts\cdots\ts S_\s^{i_n}(V).$$
From the definition of differentials $d$ and $\delta$, they both preserve this gradation on the Bar and coBar complexes and give gradations on the Hochschild and coHochschild homology groups; we let ${\HH}_n(S_\s(V),S_{\widetilde{\s}}(W)_{(1)})_t$ and $\Hoch^n(S_\s(V),S_{\widetilde{\s}}(W)_{(1)})_t$ denote sets of homology classes of length degree $t$ and homology degree $n$.
\par
Fixing some degree $t$, there are subcomplexes $C_\bullet\ra S_{\widetilde{\s}}^t(W)_{(1)}$ and $S_{\widetilde{\s}}^t(W)_{(1)}\ra C_\bullet$ defined by:
$$C_r=\bigoplus_{i_0+\cdots+i_r=t}S_{\widetilde{\s}}^{i_0}(W)_{(1)}\ts S_\s^{i_1}(V)\ts\cdots\ts S_\s^{i_r}(V);$$
these $C_r$ are finite dimensional.
\par
This subcomplex satisfies conditions in Proposition \ref{duality}, so the Hopf pairing induces an isomorphism of vector space: for any $n,t=0,1,\cdots,$
$${\HH}_n(S_\s(V),S_{\widetilde{\s}}(W)_{(1)})_t\cong {\Hoch}^n(S_\s(V),S_{\widetilde{\s}}(W)_{(1)})_t.$$

\subsection{De Concini-Kac filtration}

Let $S$ be a commutative totally ordered semi-group and $A$ be an $S$-filtered algebra with unit, that is to say: $A=\bigcup_{s\in S} A_s$ such that
\begin{enumerate}
\item for any $s\in S$, $A_s$ is a subspace of $A$;
\item for any $s<s'\in S$, $A_s\subset A_{s'}$;
\item for any $s,s'\in S$, $A_s\cdot A_{s'}\subset A_{s+s'}$.
\end{enumerate}
The graded algebra associated to this filtration is denoted by $\gr A=\bigoplus_{s\in S}\gr_sA$, where $\gr_sA=A_s/\sum_{s'<s}A_{s'}$.
\par
Let $M$ be a free left $A$-module with a generating set $M_0$. A filtration of $A$ induces a filtration on $M$ by defining $M_s=A_s\cdot M_0$. Then $\{M_s\}_{s\in S}$ forms a filtration on $M$ which is compatible with the $A$-module structure: for any $s,s'\in S$, $A_s.M_{s'}\subset M_{s+s'}$. We let $\gr M$ denote the associated graded vector space $\bigoplus_{s\in S}\gr_sM$, where $\gr_sM=M_s/\sum_{s'<s}M_{s'}$. Then $\gr M$ is a left $\gr A$-module.\\
\\
\noindent
\textbf{Convention.} From now on until the end of this paper, $\g$ is assumed to be a finite dimensional simple Lie algebra.
\par
We study the De Concini-Kac filtration in the rest part of this subsection.
\par
We fix a total order $F_1<\cdots<F_n$ on the basis (alphabets) $F_1,\cdots,F_n$ of $V$. This allows us to construct the PBW basis of $S_\s(V)$ using Lyndon words: there is a bijection between the set of good Lyndon words in $S_\s(V)$ and the set of the positive roots $\Delta_+$ of $\g$ (see Theorem 22, \cite{Ros02}). The order on the set of alphabets induces lexicographically a convex order on good Lyndon words and then positive roots of $\g$: we let $\Delta_+=\{\beta_1,\cdots,\beta_N\}$ and denote this order by $\beta_1>\cdots>\beta_N$. For each $\beta_i\in\Delta_+$, we let $F_{\beta_i}\in S_\s(V)$ denote the PBW root vector associated to the good Lyndon word corresponding to $\beta_i\in\Delta_+$, then the set
$$\{F_{\beta_1}^{i_1}\cdots F_{\beta_N}^{i_N}|\ (i_1,\cdots,i_N)\in\mathbb{N}^N\}$$
 forms a linear basis of $S_\s(V)$ which can be identified with the lattice $\mathbb{N}^N$. We equip $\mathbb{N}^N$ with its lexicographical order, then it is a totally ordered commutative semi-group. For a monomial $F_{\beta_1}^{i_1}\cdots F_{\beta_N}^{i_N}$, we define its degree 
$$d(F_{\beta_1}^{i_1}\cdots F_{\beta_N}^{i_N})=(i_1,\cdots,i_N)\in\mathbb{N}^N.$$
If $\underline{i}=(i_1,\cdots,i_N)$, the notation $F^{\underline{i}}=F_{\beta_1}^{i_1}\cdots F_{\beta_N}^{i_N}$ will be adopted.
\par
The following lemma us due to Levendorskii-Soibelman and Kirillov-Reshetikhin.
\begin{lemma}
For any $\beta_i<\beta_j$, we have:
$$F_{\beta_j}F_{\beta_i}-q^{(\beta_i,\beta_j)}F_{\beta_i}F_{\beta_j}=\sum_{\underline{k}\in\mn^N}\alpha_{\underline{k}}F^{\underline{k}},$$
where $\alpha_{\underline{k}}\in k$ and $\alpha_{\underline{k}}\neq 0$ unless $d(F^{\underline{k}})<d(F_{\beta_i}F_{\beta_j})$.
\end{lemma}

We define an $\mathbb{N}^N$-filtration on $S_\s(V)$ by: for $\underline{i}\in\mathbb{N}^N$, $S_\s(V)_{\underline{i}}$ is the linear subspace of $S_\s(V)$ generated by monomials $F^{\underline{k}}$ such that $d(F^{\underline{k}})\leq\underline{i}$.

\begin{proposition}[De Concini-Kac, \cite{Con90}]
\ 
\begin{enumerate}
\item $\{S_\s(V)_{\underline{i}}|\ \underline{i}\in\mathbb{N}^N\}$ forms an $\mathbb{N}^N$-filtration of $S_\s(V)$.
\item The associated graded algebra $\gr S_\s(V)$ is generated by homogeneous generators $\{F_{\beta_i}|\ i=1,\cdots,N\}$ and relations:
$$F_{\beta_j}F_{\beta_i}=q^{(\beta_i,\beta_j)}F_{\beta_i}F_{\beta_j},\ \ \ \text{for}\ \beta_i<\beta_j.$$
\end{enumerate}
\end{proposition}

That is to say, $\gr S_\s(V)$ is a kind of "multi-parameter quantum plane"; it is an integral algebra, i.e., has no nontrivial zero-divisors.
\par
Now we turn to study the induced $\mathbb{N}^N$-filtration on the left $S_\s(V)$-module $S_{\wt{\s}}(W)_{(1)}$. Recall that there exists an isomorphism of vector space 
$$S_{\wt{\s}}(W)_{(1)}\cong L(\lambda)\ts S_\s(V).$$
Let $v_1,\cdots,v_r$ be a linear basis of $L(\lambda)$, where $r=\dim L(\lambda)$. Then a linear basis of $S_{\wt{\s}}(W)_{(1)}$ is given by 
$$F_{\beta_1}^{i_1}\cdots F_{\beta_N}^{i_N}v_1^{\ve_1}\cdots v_r^{\ve_r},$$
where $\underline{i}=(i_1,\cdots,i_N)\in\mn^N$ and $\underline{\ve}=(\ve_1,\cdots,\ve_r)\in\{0,1\}^r$ such that $|\ve|=\sum_{i=1}^r\ve_i=1$.
\par
The induced $\mathbb{N}^N$-filtration on $M=S_{\wt{\s}}(W)_{(1)}$ can be determined as follows: for any $\underline{s}\in\mn^N$,
$$M_{\underline{s}}=\{F_{\beta_1}^{i_1}\cdots F_{\beta_N}^{i_N}v_1^{\ve_1}\cdots v_r^{\ve_r}|\ \underline{i}\leq \underline{s}\}.$$
We let $\gr M=\gr S_{\wt{\s}}(W)_{(1)}$ denote the associated $\gr S_\s(V)$-module.
\par
We have an explicit description of $\gr M$ according Theorem \ref{RossoThm}:
$$\gr M=\bigoplus_{\underline{i}\in\mathbb{N}^N, l=1,\cdots,r}\mc\ F_{\beta_1}^{i_1}\cdots F_{\beta_N}^{i_N}\ts v_l,$$
where the $\gr S_\s(V)$-module structure is given by: for $F_{\beta_t}\in \gr S_\s(V)$ and $F_{\beta_1}^{i_1}\cdots F_{\beta_N}^{i_N}\ts v_k\in \gr M$,
$$F_{\beta_t}\cdot F_{\beta_1}^{i_1}\cdots F_{\beta_N}^{i_N}\ts v_k=\prod_{s=1}^{t-1}q^{-i_s(\beta_s,\beta_t)}F_{\beta_1}^{i_1}\cdots F_{\beta_t}^{i_t+1}\cdots F_{\beta_N}^{i_N}\ts v_k.$$
\par
The $S_\s(V)$-bimodule structure on $S_{\widetilde{\s}}(W)_{(1)}$ gives $\gr S_{\wt{\s}}(W)_{(1)}$ a $\gr S_\s(V)$-bimodule structure.
\par
This construction still works when $q^l=1$ is a root of unity, we refer to \cite{Gin93} for complete statements.

\subsection{Hochschild homology of graded algebra: generic case}\label{section}
This section is devoted to computing the Hochschild homology group $\HH_\bullet(\gr S_\s(V),\gr S_{\wt{\s}}(W)_{(1)})$ with the bimodule structure defined above. To simplify the notation, we let $G_\s(V)$ and $G_{\wt{\s}}(W)_{(1)}$ denote $\gr S_\s(V)$ and $\gr S_{\wt{\s}}(W)_{(1)}$ respectively.
\par
The main theorem of this section is:
\begin{theorem}
Let $\lambda\in\mathcal{P}_{+}$ and $q$ not be a root of unity. Then the Hochschild homology group of $G_\s(V)$ with coefficient in $G_{\wt{\s}}(W)_{(1)}$ is given by:
$${\HH}_n(G_\s(V),G_{\wt{\s}}(W)_{(1)})=\left\{\begin{matrix}L(\lambda) & n=0;\\ 
0, & n\neq 0.\end{matrix}\right.
$$
\end{theorem}

The main idea of the proof is using the Koszul resolution of the Koszul algebra $G_\s(V)$, then apply an analogue of the homotopy defined by M. Wambst in \cite{Wam93}.\\
\indent
We let $\Lambda_q(V)$ denote the algebra generated by homogeneous generators $F_{\beta_1},\cdots, F_{\beta_N}$ with graded degree $1$ for each generator and relations
\begin{enumerate}
\item for any $\beta_i<\beta_j$, $F_{\beta_j}F_{\beta_i}+q^{(\beta_i,\beta_j)}F_{\beta_i}F_{\beta_j}=0$;
\item for any $i=1,\cdots,N$, $F_{\beta_i}^2=0$.
\end{enumerate}

Then $\Lambda_q(V)=\bigoplus_{k=0}^N \Lambda_q^k(V)$, where $\Lambda_q^k(V)$ is generated as a vector space by $F_{\beta_{i_1}}\wedge\cdots\wedge F_{\beta_{i_k}}$ for $i_1<\cdots<i_k$. The algebra $\Lambda_q(V)$ is the Koszul dual of $G_\s(V)$.
\par
According to Theorem 5.3 in \cite{Pri70}, $G_\s(V)$ is a homogeneous Koszul algebra as it is obviously a PBW algebra (see Section 5.1 of \cite{Pri70} for a definition). Then as a $G_\s(V)$-bimodule, there is a Koszul complex starting from $G_{\wt{\s}}(W)_{(1)}$:
\[
\xymatrix{
\cdots \ar[r] & G_{\wt{\s}}(W)_{(1)}\ts \Lambda_q^k(V) \ar[r]^-d & \cdots \ar[r]^-d & 
G_{\wt{\s}}(W)_{(1)}\ts \Lambda_q^1(V) \ar[r]^-d & G_{\wt{\s}}(W)_{(1)}.
}
\]
We write down the differential $d$ explicitly: for $\underline{i}\in\mn^N$ and $\underline{\ve}\in\{0,1\}^r$ with $|\underline{\ve}|=1$, we denote
$$F^{(\underline{i},\underline{\ve})}=F_{\beta_1}^{i_1}\cdots F_{\beta_N}^{i_N}v_1^{\ve_1}\cdots v_r^{\ve_r},$$
then 
$$d(F^{(\underline{i},\underline{\ve})}\ts F_{\beta_{i_1}}\wedge\cdots\wedge F_{\beta_{i_n}})=\sum_{k=1}^n (-1)^{k-1}\prod_{s=k+1}^n Q_{i_ki_s}F_{\beta_{i_k}}F^{(\underline{i},\underline{\ve})}\ts F_{\beta_1}\wedge\cdots\wedge \widehat{F_{\beta_{i_k}}}\wedge\cdots\wedge F_{\beta_{i_n}},$$
where $Q_{i_ki_s}=q^{(\beta_{i_k},\beta_{i_s})}$ and the hat notation hides the corresponding term. The linear map $d$ is well-defined since the right $G_\s(V)$-module on $G_{\wt{\s}}(W)_{(1)}$ is trivial.
\par
To simplify notations, we let $Q_{ij}$ denote the number such that $F_{\beta_i}F_{\beta_j}=Q_{ij}F_{\beta_j}F_{\beta_i}$ for any positive roots $\beta_i,\beta_j$. Then relations in $G_\s(V)$ imply that $Q_{ij}=Q_{ji}^{-1}$ for any $i\neq j$ and $Q_{ii}=1$.
\par
Now we mimic the definition of the homotopy in \cite{Wam93}, Section 6 to show that the Koszul complex above is acyclic.
\par
For any $\alpha=(\underline{i},\underline{\ve})\in\mn^N\times \{0,1\}^r$ as above with $|\underline{\ve}|=1$ and any $\beta\in\{0,1\}^N$, we define for any $i=1,\cdots,N$,

$$\Omega(\alpha,\beta,i)=\left\{\begin{matrix}0 & \beta_i=0;\\ 
\D\ve(\beta,i)\prod_{s=i+1}^N Q_{is}^{\beta_s}\prod_{p=1}^{i-1}Q_{pi}^{-i_p}, & \beta_i\neq 0,\end{matrix}\right.
$$
where $\ve(\beta,i)=(-1)^{\sum_{s=1}^{i-1}\beta_s}$.
\par
For any $i=1,\cdots,N$, we let $[i]$ denote the element in $\mn^N\times\{0,1\}^r$ or in $\{0,1\}^N$ such that its i-th component is $1$ and the others are zero. Then the differential in the Koszul complex can be written as:
$$d(F^\alpha\ts F^\beta)=\sum_{i=1}^N \Omega(\alpha,\beta,i)F^{\alpha+[i]}\ts F^{\beta-[i]}.$$
With notations above, we define
$$\omega(\alpha,\beta,i)=\left\{\begin{matrix}0 & \beta_i=1\ \ \text{or}\ \ \alpha_i=0;\\ 
\Omega(\alpha-[i],\beta+[i],i)^{-1}, & \text{if not}.\end{matrix}\right.
$$
We define a map $h:G_{\wt{\s}}(W)_{(1)}\ts \Lambda_q^n(V)\ra G_{\wt{\s}}(W)_{(1)}\ts \Lambda_q^{n+1}(V)$ by:
$$h(F^\alpha\ts F^\beta)=\frac{1}{||\alpha+\beta||}\sum_{i=1}^N\omega(\alpha,\beta,i)F^{\alpha-[i]}\ts F^{\beta+[i]},$$
where $||\alpha+\beta||=\textrm{Card}(\{i=1,\cdots,N|\ (\alpha+\beta)_i\neq 0\}$) (here we extend $\beta$ by $0$ to an element in $\mn^N\times\{0,1\}^r\subset \mn^{N+r}$ and $(\alpha+\beta)_i$ is the i-th component of $\alpha+\beta$).

\begin{lemma}
$hd+dh=1$, i.e., $h$ is a homotopy.
\end{lemma}

A similar argument as in the proof of Theorem 6.1 of \cite{Wam93} can be applied to our case to prove this lemma. We provide in the end of this subsection some details of this verification since a modified version will be applied to the root of unity case. Thus the complex $G_{\wt{\s}}(W)_{(1)}\ts \Lambda_q^\bullet(V)$ is acyclic.
\par
It remains to compute the homology group degree $0$. This can be directly shown as follows: by definition, the degree $0$ part is $G_{\wt{\s}}(W)_{(1)}/\text{im}d$ where 
$$d:G_{\wt{\s}}(W)_{(1)}\ts V\ra G_{\wt{\s}}(W)_{(1)}$$
is given by: 
$$d(F^\alpha\ts F_{\beta_i})=F_{\beta_i}F^\alpha.$$ 
Recall that we have a linear basis $F^{(\underline{i},\underline{\ve})}$ for $G_{\wt{\s}}(W)_{(1)}$ where $\underline{i}\in\mn^N$ and $\underline{\ve}\in\{0,1\}^r$ satisfying $|\underline{\ve}|=1$. So after the quotient, all elements surviving are those $F^{(\underline{i},\underline{\ve})}$ with $\underline{i}=(0,\cdots,0)$. Thus $\HH_0(G_\s(V),G_{\wt{\s}}(W)_{(1)})=L(\lambda)$ as vector space.\qed\\

\begin{proof}[Proof of Lemma]
It suffices to verify that $h$ is a homotopy. According to the explicit formulas of $d$ and $h$, 
$$dh(F^\alpha\ts F^\beta)=\frac{1}{||\alpha+\beta||}\sum_{i=1}^N\sum_{j=1}^N
\omega(\alpha,\beta,i)\Omega(\alpha-[i],\beta+[i],j)F^{\alpha-[i]+[j]}\ts F^{\beta+[i]-[j]},
$$
$$hd(F^\alpha\ts F^\beta)=\frac{1}{||\alpha+\beta||}\sum_{i=1}^N\sum_{j=1}^N
\Omega(\alpha,\beta,j)\omega(\alpha+[j],\beta-[j],i)F^{\alpha-[i]+[j]}\ts F^{\beta+[i]-[j]}.
$$
We want to show that for $i\neq j$,
$$\omega(\alpha,\beta,i)\Omega(\alpha-[i],\beta+[i],j)+\Omega(\alpha,\beta,j)\omega(\alpha+[j],\beta-[j],i)=0.$$
If $i<j$, from the definition of $\Omega$ and $\omega$, we have:
\begin{eqnarray*}
& &\ve(\beta,i)\ve(\beta+[i],j)\left(\prod_{s=j+1}^NQ_{js}^{\beta_s}\prod_{p=1}^{j-1} Q_{pj}^{-i_p}\right)\left(\prod_{s=i+1}^N Q_{is}^{\beta_s}Q_{ij}^{-1}\prod_{p=1}^{i-1}Q_{pi}^{-i_p}\right)^{-1}\\
&+&\ve(\beta,j)\ve(\beta-[j],i)\left(\prod_{s=j+1}^N Q_{js}^{\beta_s}\prod_{p=1}^{j-1} Q_{pj}^{-i_p}Q_{ij}\right)\left(\prod_{s=i+1}^N Q_{is}^{\beta_s}\prod_{p=1}^{i-1}Q_{pi}^{-i_p}\right)^{-1}=0
\end{eqnarray*}
Therefore it suffices to show that if $i<j$,
$$\ve(\beta,i)\ve(\beta+[i],j)+\ve(\beta,j)\ve(\beta-[j],i)=0,$$
but this is clear from definition. Moreover, the case $i>j$ can be similarly tackled.
\par
It suffices to prove that 
$$\sum_{i=1}^N\left(
\omega(\alpha,\beta,i)\Omega(\alpha-[i],\beta+[i],i)+\Omega(\alpha,\beta,i)\omega(\alpha+[i],\beta-[i],i)\right)=||\alpha+\beta||.$$
Notice that each multiplication of $\omega$ and $\Omega$ is either $0$ or $1$. So we separate them into four cases: 
\begin{enumerate}
\item $\alpha_i=0$, $\beta_i=0$; in this case, the $i$-summand is $0$;
\item $\alpha_i\neq 0$, $\beta_i=0$; in this case, the $i$-summand is $1$;
\item $\alpha_i=0$, $\beta_i=1$; in this case, the $i$-summand is $1$;
\item $\alpha_i\neq 0$, $\beta_i=1$; in this case, the $i$-summand is $1$.
\end{enumerate}
Thus $\omega(\alpha,\beta,i)\Omega(\alpha-[i],\beta+[i],i)+\Omega(\alpha,\beta,i)\omega(\alpha+[i],\beta-[i],i)=1$ if and only if $(\alpha+\beta)_i\neq 0$, from which the identity above.
\end{proof}

\subsection{Hochschild homology of graded algebra: root of unity case}
We let $q^l=1$ be a primitive $l$-th root of unity in this subsection.

\begin{theorem}
Let $\lambda\in\mathcal{P}_{+}^l$. The Hochschild homology groups of $G_\s(V)$ with coefficients in $G_{\wt{\s}}(W)_{(1)}$ are given by:
$${\HH}_n(G_\s(V),G_{\wt{\s}}(W)_{(1)})=\left\{\begin{matrix}L(\lambda) & n=0;\\ 
\wedge^n(\mathfrak{n_-}), & n\neq 1.\end{matrix}\right.
$$
where $\mathfrak{n}_-$ is identified with the negative part of the Lie algebra $\g$.
\end{theorem}

We will prove this theorem in the rest of this subsection.
\par
We use the Koszul complex as in the generic case:
$$K_\bullet=\xymatrix{
\cdots \ar[r] & G_{\wt{\s}}(W)_{(1)}\ts \Lambda_q^k(V) \ar[r]^-d & \cdots \ar[r]^-d & 
G_{\wt{\s}}(W)_{(1)}\ts \Lambda_q^1(V) \ar[r]^-d & G_{\wt{\s}}(W)_{(1)},}$$
and consider the following subcomplex of $K_\bullet$:
\begin{tiny}
$$S_\bullet=\xymatrix{
\cdots\ar[r] & \D\bigoplus_{1\leq i_1<\cdots<i_k\leq N} \mc F_{\beta_{i_1}}^{l-1}\cdots F_{\beta_{i_k}}^{l-1}\ts F_{\beta_{i_1}}\wedge\cdots\wedge F_{\beta_{i_k}}
 \ar[r]^-d & \cdots \ar[r]^-d & 
\D\bigoplus_{s=1}^N  \mc F_{\beta_{i_s}}^{l-1}\ts F_{\beta_{i_s}} \ar[r]^-d & 0}.$$
\end{tiny}

From the definition of the differential and the fact that for any $1\leq i\leq n$, $F_{\beta_i}^l=0$, each term in $S_\bullet$ has no pre-image under $d$ in $K_\bullet$. Thus we obtain a complement $R_\bullet$ of $S_\bullet$ in $K_\bullet$ such that as complexes,
$$K_\bullet=S_\bullet\oplus R_\bullet.$$
It is clear that in $S_\bullet$, all differentials are zero, so identifying $F_{\beta_{i_1}}^{l-1}\cdots F_{\beta_{i_k}}^{l-1}\ts F_{\beta_{i_1}}\wedge\cdots\wedge F_{\beta_{i_k}}$ with $F_{\beta_{i_1}}\wedge\cdots\wedge F_{\beta_{i_k}}$ gives a bijection: for $k\geq 1$, 
$$\xymatrix{{\HH}_k(S_\bullet)\ar[r]^-\sim & \wedge^k(\mathfrak{n}_-)}.$$
Now we proceed to show that the complex $R_\bullet$ is acyclic with $\HH_0(R_\bullet)=L(\lambda)$ by applying a modification of the homotopy defined in the generic case.
\par
We explain the modifications:
\begin{enumerate}
\item The definition of $\Omega(\alpha,\beta,j)$ for $j=1,\cdots,N$, $\alpha=(\underline{i},\underline{\ve})\in(\mathbb{Z}/l)^N\times \{0,1\}^r$, $\beta\in\{0,1\}^N$ with $|\underline{\ve}|=1$:
$$\Omega(\alpha,\beta,j)=\left\{\begin{matrix}0 & \beta_j=0\ \text{or}\ \underline{i}_j=l-1;\\ 
\D\ve(\beta,j)\prod_{s=j+1}^NQ_{js}^{\beta_s}\prod_{p=1}^{j-1}Q_{pi}^{-i_p}, & \beta_j\neq 0\ \text{and}\ \underline{i}_j\neq l-1,\end{matrix}\right.
$$
\item The definition of $\omega(\alpha,\beta,j)$ need not to be changed.
\item In the last step of the proof, there are 6 cases to be considered:
\begin{enumerate}
\item $\alpha_i=0$, $\beta_i=0$, then the i-summand is $0$;
\item $\alpha_i=0$, $\beta_i=1$, then the i-summand is $1$;
\item $\alpha_i=l-1$, $\beta_i=0$, then the i-summand is $1$;
\item $\alpha_i=l-1$, $\beta_i=1$, then the i-summand is $0$;
\item $\alpha_i\neq 0,l-1$, $\beta_i=0$, then the i-summand is $1$;
\item $\alpha_i\neq 0,l-1$, $\beta_i=1$, then the i-summand is $1$.
\end{enumerate}
So the $i$-summand is $0$ if and only if $(\alpha+\beta)_i=0$ in $(\mathbb{Z}/l)^N\times\{0,1\}^r$. Moreover, in the complex $R_\bullet$, there does not exist a term $F^\alpha\ts F^\beta$ such that for any $i$, $(\alpha+\beta)_i=0$ (since such terms are all contained in $S_\bullet$), which implies that $||\alpha+\beta||\neq 0$. So a homotopy from $R_\bullet$ to itself can be defined similarly as in the generic case, which shows that $R_\bullet$ is acyclic.
\end{enumerate}
Finally, the result concerning $\HH_0(R_\bullet)$ can be proved as in the generic case, which terminates the proof.

\subsection{Main results}\label{mainresult}
Theorems in the last subsections permit us to compute the Hochschild homology groups $\HH_\bullet(S_\s(V),S_{\wt{\s}}(W)_{(1)})$ by an argument of spectral sequence.

\begin{theorem}\label{Thm6.7}
The Hochschild homology groups of $S_\s(V)$ with coefficient in the $S_\s(V)$-bimodule $S_{\wt{\s}}(W)_{(1)}$ are:
\begin{enumerate}
\item If $q$ is not a root of unity and $\lambda\in\mathcal{P}_{+}$, we have:
$${\HH}_n(S_\s(V),S_{\wt{\s}}(W)_{(1)})=\left\{\begin{matrix}L(\lambda) & n=0;\\ 
0, & n\neq 0.\end{matrix}\right.
$$
\item If $q^l=1$ is a primitive root of unity and $\lambda\in\mathcal{P}_{+}^l$, we have:
$${\HH}_n(S_\s(V),S_{\wt{\s}}(W)_{(1)})=\left\{\begin{matrix}L(\lambda) & n=0;\\ 
\wedge^n(\mathfrak{n_-}), & n\geq 1.\end{matrix}\right.
$$
where $\mathfrak{n}_-$ is identified with the negative part of the Lie algebra $\g$.
\end{enumerate}
\end{theorem}
To pass from the graded case to the general case, it suffices to apply the following lemma due to May \cite{May66}, Theorem 3.

\begin{lemma}[May spectral sequence]
Let $A$ be a filtered algebra with unit such that its filtration is exhaustive, $M$ be a filtered $A$-module where the filtration is induced from that of $A$. Then there exists a convergent spectral sequence
$$E_{p,q}^2={\HH}_{p+q}(\gr A,\gr M)\Longrightarrow {\HH}_\bullet(A,M). $$
\end{lemma}

In our context, from theorems in last subsections, the spectral sequence collapses at $E^2$-term: it is clear in the generic case; in the root of unity case, this holds as all differentials are zero in the $E^2$-sheet. This gives the desired isomorphism of homology groups:
\[
\xymatrix{{\HH}_n(G_\s(V),G_{\wt{\s}}(W)_{(1)})\ar[r]^-\sim & {\HH}_n(S_\s(V),S_{\wt{\s}}(W)_{(1)})},
\]
which finishes the proof.
\par
As a corollary, Theorem A announced in the introduction comes from Proposition \ref{dual}. 

\section{On the study of coinvariants of degree 2}\label{Section:degree2}

From now on until the end of this paper, we suppose that $q$ is not a root of unity and $\lambda\in\mathcal{P}_{+}$ is a dominant integral weight.

In this section, as a continuation of Theorem \ref{RossoThm}, we will study the set of coinvariants of degree 2. We keep notations in last sections.
\par
We fix an integer $n\geq 1$. Let $M_n$ denote the subspace of $S_{\wt{\sigma}}(W)$ containing elements of degree $n$ with the degree structure defined by $\deg(F_i)=0$ and $\deg(v_\lambda)=1$. The same argument as in last sections shows that $M_n=(S_{\wt{\sigma}}(W))_{(n)}$ is an $S_\s(V)$-sub-Hopf bimodule of $S_{\wt{\sigma}}(W)$. Then we can consider the set of right coinvariants $M_n^{coR}$ in $M_n$ and from the structure theorem of Hopf bimodules,
$$S_{\wt{\sigma}}(W)_{(n)}\cong M_n^{coR}\ts S_\s(V).$$

\subsection{Basic construction}
We will concentrate on the case $n=2$ to give an explicit description of $M_2^{coR}$. The main tool for tackling this case is the following construction.
\par
Consider the following commutative diagram:
\[
\xymatrix{
T(W)_{(1)}\ts_{T(V)}T(W)_{(1)}
\ar[r]^-{S_1} \ar[d]_-{\cong}^-m  & S_{\wt{\sigma}}(W)_{(1)}\ts_{S_\s(V)}S_{\wt{\sigma}}(W)_{(1)} \ar[d]^-m\\
T(W)_{(2)}\ar[r]^-{S_2} & S_{\wt{\sigma}}(W)_{(2)}.
}
\]
We start from explaining morphisms appearing in this diagram. Recall that $\Sigma_n=\sum_{\s\in\mathfrak{S}_n}T_\s\in \mc[\mathfrak{B}_n]$ is the total symmetrization operator: it acts linearly on $V^{\ts n}$. The map 
$$S_1=\bigoplus_{n,m=0}^\infty\Sigma_n\ts \Sigma_m$$
is given by the symmetrization on both components. It is well-defined since as explained in Section \ref{Section:QSA}, the symmetrization map $T(W)\ra S_{\wt{\s}}(W)$ is an algebra morphism. The morphism $S_2=\bigoplus_{n=0}^\infty \Sigma_n$ is the symmetrization map. The horizontal morphisms are given by symmetrization, so both of them are surjection. Two vertical morphisms are given by multiplications, then the left one is an isomorphism. It permits us to identify elements in $T(W)_{(1)}\ts_{T(V)}T(W)_{(1)}$ and $T(W)_{(2)}$. We will denote $S_2\circ m$ by $S_2$ for short.

\begin{lemma}
We have $\ker S_1\subset \ker S_2$, so the right vertical map $m$ is surjective.
\end{lemma}

\begin{proof}[Proof]
This comes from a general observation. For three integers $n,n_1,n_2\geq 0$ satisfying $n=n_1+n_2$, as shown in Section \ref{Section:symmbraid}, we can decompose the symmetric group $\mathfrak{S}_n$ as
$$\mathfrak{S}_n=(\mathfrak{S}_{n_1}\times \mathfrak{S}_{n_2})\circ\mathfrak{S}_{n_1,n_2}.$$
Moreover, this decomposition can be lifted to the braid group $\mathfrak{B}_n$ by the Matsumoto section. So for an element $x$ in $W^{\ts n}$, if
$$\sum_{\s\in\mathfrak{S}_{n_1,n_2}} T_\s(x)=0,$$
the total symmetrization
$$\sum_{\omega\in\mathfrak{S}_n}T_\omega(x)=\sum_{\tau\in\mathfrak{S}_{n_1}\times \mathfrak{S}_{n_2}}T_\tau\left(\sum_{\s\in\mathfrak{S}_{n_1,n_2}} T_\s(x)\right)=0.$$
From the definition of $S_1$ and $S_2$, it is clear that $\ker S_1\subset\ker S_2$. Thus we obtain a linear surjection
$$S_{\wt{\sigma}}(W)_{(1)}\ts_{S_\s(V)}S_{\wt{\sigma}}(W)_{(1)}\ra S_{\wt{\sigma}}(W)_{(2)}$$
given by the multiplication.
\end{proof}

Now we consider the inclusion map 
$$S_{\widetilde{\s}}(W)_{(1)}\ra S_{\widetilde{\s}}(W)$$
which is a regular $S_\s(V)$-bimodule morphism. Thanks to the universal property of the tensor algebra (see, for example, Proposition 1.4.1 of \cite{Nic78}), it can be lifted to an algebra morphism
$$T_{S_\s(V)}(S_{\widetilde{\s}}(W)_{(1)})\ra S_{\widetilde{\s}}(W)$$
given by the multiplication.
\par
A similar argument as in the lemma above can be applied to show the following corollary:
\begin{corollary}
The multiplication map $T_{S_\s(V)}(S_{\widetilde{\s}}(W)_{(1)})\ra S_{\widetilde{\s}}(W)$ is surjective.
\end{corollary}
That is to say, as an $S_\s(V)$-bimodule, $S_{\widetilde{\s}}(W)$ is generated by $S_{\widetilde{\s}}(W)_{(1)}$.
\par
We proceed to consider the $S_\s(V)$-structures of this morphism.
\begin{enumerate}
\item $T_{S_\s(V)}(S_{\widetilde{\s}}(W)_{(1)})$ is an $S_\s(V)$-bicomodule: for each $n\in\mathbb{N}$, as a tensor product, $T^n(S_{\widetilde{\s}}(W)_{(1)})$ is an $S_\s(V)$-bicomodule. Combined with the canonical projection, we obtain linear maps
$$\delta_L:T^n(S_{\widetilde{\s}}(W)_{(1)})\ra S_\s(V)\ts T^n_{S_\s(V)}(S_{\widetilde{\s}}(W)_{(1)}),$$
$$\delta_R:T^n(S_{\widetilde{\s}}(W)_{(1)})\ra T^n_{S_\s(V)}(S_{\widetilde{\s}}(W)_{(1)})\ts S_\s(V)$$
which make $T^n_{S_\s(V)}(S_{\widetilde{\s}}(W)_{(1)})$ an $S_\s(V)$-bicomodule, according to Lemma \ref{Lem:tensor}.
\item $T_{S_\s(V)}(S_{\widetilde{\s}}(W)_{(1)})$ is an $S_\s(V)$-bimodule as each $T^n_{S_\s(V)}(S_{\widetilde{\s}}(W)_{(1)})$ is.
\item These two structures on $T_{S_\s(V)}(S_{\widetilde{\s}}(W)_{(1)})$ make it into an $S_\s(V)$-Hopf bimodule.
\item $S_{\widetilde{\s}}(W)$ has its ordinary $S_\s(V)$-bimodule and bicomodule structures as in the beginning of Section \ref{HopfBim}.
\item As right $S_\s(V)$-Hopf modules, there are isomorphisms
$$M_n^{coR}\ts S_\s(V)\cong S_{\widetilde{\s}}(W)_{(n)},\,\, M^{coR}\ts S_\s(V)\cong S_{\widetilde{\s}}(W)$$
given by multiplications.
\end{enumerate}

Combine these observations and constructions together, we have: as left $S_\s(V)$-modules and comodules,
$$T_{S_\s(V)}(S_{\widetilde{\s}}(W)_{(1)})\cong T(M_1^{coR})\ts S_\s(V),$$
$$S_{\widetilde{\s}}(W)\cong \bigoplus_{n=0}^\infty M_n^{coR}\ts S_\s(V).$$
These $M_n^{coR}$ and $T(M_1^{coR})$ are $S_\s(V)$-Yetter Drinfel'd modules as explained in Section \ref{Hopfbimod}. Once $\id\ts\ve$ is applied to both sides, we obtain a surjection of $S_\s(V)$-Yetter-Drinfel'd modules $T(M_1^{coR})\ra \bigoplus_{n=0}^\infty M_n^{coR}$ given by the multiplication. The following corollary is just a particular case:

\begin{corollary}
The multiplication gives an $S_\s(V)$-Yetter-Drinfel'd module surjection
$$m:M_1^{coR}\ts M_1^{coR} \ra M_2^{coR}.$$
\end{corollary}

\subsection{Study of $M_2^{coR}$: non-critical case}
The non-critical and critical case are separated in view of studying new Serre relations appearing when passing from degree 1 to degree 2.
\par
Let $\lambda$ be a dominant weight. We call $\lambda$ non-critical in degree 2 if there does not exist $i\in I$ such that $(\lambda,\alpha_i)=1$. If the degree is under consideration, we will call $\lambda$ non-critical.
\par
We start from a general remark. Let $C'=(c'_{i,j})_{(n+1)\times(n+1)}$be a generalized Cartan matrix obtained from $C$ by adding a last row and a last column, whose elements are: for $1\leq i,j\leq n$, $c'_{i,j}=c_{i,j}$; $c'_{n+1,n+1}=2$ and for $1\leq i\leq n$, $c'_{n+1,i}=c'_{i,n+1}=-(\lambda,\alpha_i)$.\\
\indent
Thanks to Theorem \ref{RosMain}, $S_{\wt{\sigma}}(W)\cong U_q^-(\g(C'))$ as braided Hopf algebra.\\
\indent
From the definition of the quantized enveloping algebra, $U_q^-(\g(C'))$, as an algebra, is generated by $F_1,\cdots,F_n$ and $v_\lambda$ with relations:
$$\ad(F_i)^{1-a_{ij}}(F_j)=0,\ \ i\neq j=1,\cdots,n;$$
$$\ad(v_\lambda)^{1+(\lambda,\alpha_i)}(F_i)=0,\ \ \ad(F_i)^{1+(\lambda,\alpha_i)}(v_\lambda)=0,\ \ \text{for}\  i=1,\cdots,n.$$
\indent
The following theorem determines the set of coinvariants of degree $2$ in the non-critical case.

\begin{theorem}\label{Thm:noncritical}
Suppose that for any $i\in I$, $(\lambda,\alpha_i)\neq 1$. Then the multiplication map gives an isomorphism of left $S_\s(V)$-Yetter-Drinfel'd modules
\[
\xymatrix{
L(\lambda)\ts L(\lambda)\ar[r]^-\sim & M_2^{coR}.
}
\]
\end{theorem}
\begin{proof}[Proof]
It suffices to show that in this case, the surjection $m:S_{\wt{\s}}(W)_{(1)}\ts_{S_\s(V)}S_{\wt{\s}}(W)_{(1)}\ra S_{\wt{\s}}(W)_{(2)}$ is an isomorphism.
\par
The following lemma comes from basic linear algebra.
\begin{lemma}
Let $U,V,W$ be three vector spaces and $f:U\ra V$, $g:V\ra W$ be two linear surjections. We denote $h=g\circ f$. Then $h$ is surjective, $\ker f\subset \ker h$ and $\ker g=\ker h/\ker f$.
\end{lemma}
From this lemma, to prove that $m$ is an injection, it suffices to show that $\ker (S_2\circ m)=\ker S_1$.\\
\indent
We consider the difference between $\ker S_1$ and $\ker S_2$. From the general remark before the theorem and the fact that $\ker S_1$ and $\ker S_2$ are generated by quantized Serre relations, it suffices to compare the quantized Serre relations appearing in $S_{\wt{\sigma}}(W)_{(1)}\ts_{S_\s(V)}S_{\wt{\sigma}}(W)_{(1)}$ and $S_{\wt{\sigma}}(W)_{(2)}$.\\
\indent
It is clear that the difference may contain only relations $\ad(v_\lambda)^{1+(\lambda,\alpha_i)}(F_i)=0$. But for a monomial in $S_{\wt{\sigma}}(W)_{(2)}$, $v_\lambda$ appears twice, so the only possible relation in the difference of the kernel is given by $\ad(v_\lambda)^2(F_i)=0$ for some $i=1,\cdots,n$.\\
\indent
As we are in the non-critical case: for any $i=1,\cdots,n$, $(\lambda,\alpha_i)\neq 1$. This forbids such relations and thus there is no difference between $\ker S_1$ and $\ker S_2$.
\end{proof}

\subsection{Study of $M_2^{coR}$: critical case}
In this subsection, we will study the critical case, that is to say, there exists some $i\in I$ such that $(\lambda,\alpha_i)=1$.\\
\indent
We start from dealing with the case where there is a unique $i\in I$ such that $(\lambda,\alpha_i)=1$.

\begin{theorem}\label{critical-1}
Let $i\in I$ be the unique index such that $(\lambda,\alpha_i)=1$. We let $L(2\lambda-\alpha_i)$ denote the irreducible sub-Yetter-Drinfel'd module of $L(\lambda)\ts L(\lambda)$ of highest weight $2\lambda-\alpha_i$. Then the multiplication map gives an isomorphism of Yetter-Drinfel'd modules
\[
\xymatrix{
\left(L(\lambda)\ts L(\lambda)\right)\left/L(2\lambda-\alpha_i)\right. \ar[r]^-\sim & M_2^{coR}.
}
\]
\end{theorem}

\begin{proof}
As in the proof above, we start from considering the difference between $\ker S_2$ and $\ker S_1$. Since $(\lambda,\alpha_i)=1$, the same argument as in the last theorem shows that this difference is generated by the element $v_\lambda^2 F_i$ in $T(W)_{(2)}$, that is to say, the relation $\ad(v_\lambda)^2(F_i)=0$ in $S_{\wt{\s}}(W)_{(2)}$. (We may calculate this directly, or adopt the method given in \cite{Fang10B}, see the example in Section 6.3 therein.)
\par
We let $\widetilde{P_i}$ denote the element in $T(W)_{(1)}\ts_{T(V)} T(W)_{(1)}$ such that $m(\widetilde{P_i})=v_\lambda^2F_i$ and $P_i=S_1(\widetilde{P_i})$; it is non-zero according to the hypothesis. We let $R_i$ denote its image in $M_1^{coR}\ts M_1^{coR}$ then $m:M_1^{coR}\ts M_1^{coR}\ra M_2^{coR}$ sends $R_i$ to zero.
\par
As a summary, we have shown that the kernel of $m:M_1^{coR}\ts M_1^{coR}\ra M_2^{coR}$ is generated by $R_i$ as an $S_\s(V)$-Yetter-Drinfel'd module since $v_\lambda^2 F_i$ is the only relation in $\ker S_2/\ker S_1$.\\
\indent
Moreover, the weight of $R_i$ is $2\lambda-\alpha_i$, which is a dominant weight according to the dominance of $\lambda$ and $(\lambda,\alpha_i)=1$. Then the $S_\s(V)$-Yetter-Drinfel'd module generated by $P_i$ is isomorphic to the irreducible representation $L(2\lambda-\alpha_i)$ of $U_q(\g)$. Thus we obtained an isomorphism of $U_q(\g)$-module
\[
\xymatrix{
\left(L(\lambda)\ts L(\lambda)\right)\left/L(2\lambda-\alpha_i)\right. \ar[r]^-\sim & M_2^{coR}.
}
\]
\end{proof}

\begin{remark}\label{Rmk:coR}
Since $v_\lambda^2\in M_2^{coR}$, there always exists a copy of the highest weight representation $L(2\lambda)$ in $M_2^{coR}$.
\end{remark}

\begin{corollary}\label{Cor:degree2}
Let $J$ be the subset of $I$ containing elements $j\in I$ satisfying $(\lambda,\alpha_j)=1$. Then we have an isomorphism of $S_\s(V)$-Yetter-Drinfel'd modules
\[
\xymatrix{
\left(L(\lambda)\ts L(\lambda)\right)\left/\D\bigoplus_{j\in J}L(2\lambda-\alpha_j)\right. \ar[r]^-\sim & M_2^{coR}.
}
\]
where $L(2\lambda-\alpha_j)$ is the irreducible sub-Yetter-Drinfel'd module of $L(\lambda)\ts L(\lambda)$ generated by $R_j$ as defined in Theorem \ref{critical-1}.
\end{corollary}

\begin{proof}[Proof]
In this case, the difference between $\ker S_2$ and $\ker S_1$ is generated by $\{v_\lambda^2 F_j|\ j\in J\}\subset T(W)_{(2)}$. Moreover, since these $L(2\lambda-\alpha_j)$ are irreducible as $S_\s(V)$-Yetter-Drinfel'd modules, they intersect trivially. Thus the same argument as in the theorem above shows that the kernel of $m:M_1^{coR}\ts M_1^{coR}\ra M_2^{coR}$ is generated by $R_j$, which are defined in the above theorem, as an $S_\s(V)$-Yetter-Drinfel'd module. This gives the corollary.
\end{proof}

\subsection{Homological interpretation}
Results obtained in this section can be interpreted in the framework of coHochschild homology as follows:

\begin{theorem}\label{Thm6.10}
Let $q\in \mc^\ast$ not be a root of unity and $\lambda\in\mathcal{P}_{+}$ be a dominant integral weight.
\begin{enumerate}
\item If for any $i\in I$, $(\lambda,\alpha_i)\neq 1$, then as $U_q(\g)$-modules,
$${\Hoch}^n(S_\s(V),S_{\wt{\s}}(W)_{(2)})=\left\{\begin{matrix}L(\lambda)\ts L(\lambda) & n=0;\\ 
0, & n\neq 0.\end{matrix}\right.
$$
\item If $J$ is the subset of $I$ containing those $j\in I$ satisfying $(\lambda,\alpha_j)=1$, then as $U_q(\g)$-modules,
$${\Hoch}^n(S_\s(V),S_{\wt{\s}}(W)_{(2)})=\left\{\begin{matrix}(L(\lambda)\ts L(\lambda))/\D\bigoplus_{j\in J}L(2\lambda-\alpha_j) & n=0;\\ 
0, & n\neq 0.\end{matrix}\right.
$$
\end{enumerate}
\end{theorem}

\begin{proof}[Proof]
Results on $\Hoch^0(S_\s(V),S_{\wt{\s}}(W)_{(2)})$ come from Proposition \ref{degree0}, Theorem \ref{critical-1} and Corollary \ref{Cor:degree2}. For the vanishing of higher homology groups, arguments from Section \ref{application} to Section \ref{section} on $S_{\wt{\s}}(W)_{(1)}$ can be applied similarly to $S_{\wt{\s}}(W)_{(2)}$ as we have already known $\Hoch^0(S_\s(V),S_{\wt{\s}}(W)_{(2)})$.
\end{proof}

\section{Example}
We will adopt notations in \cite{BouLie4-6}. In $U_q(sl_n)$, we choose the dominant weight $\lambda=\varpi_{n-1}$. Then the construction above gives the strictly negative part of $U_q(sl_{n+1})$. In this case, we have a uniform description for right coinvariants of degree $p\geq 0$.

\begin{proposition}\label{sl_n}
For any integer $p\geq 0$, we have an isomorphism of $S_\s(V)$-Yetter-Drinfel'd modules:
$$S_{\wt{\sigma}}(W)_{(p)}\cong L(p\varpi_{n-1})\ts S_\s(V).$$
If we let $M$ denote $S_{\wt{\sigma}}(W)$, then
$$M^{coR}\cong\bigoplus_{p=0}^\infty L(p\varpi_{n-1}).$$
\end{proposition}

\begin{proof}
By induction on $p$, it suffices to show that the surjection 
$$m:(L((p-1)\varpi_{n-1})\ts L(\varpi_{n-1}))\left/L(p\varpi_{n-1}-\alpha_{n-1})\right.\ra M_p^{coR}$$
is bijective. We apply the following decomposition:
$$L((p-1)\varpi_{n-1})\ts L(\varpi_{n-1})\cong L(p\varpi_{n-1})\oplus L(p\varpi_{n-1}-\alpha_{n-1}).$$
\end{proof}

\begin{remark}
The principal method in this paper can be used to give an inductive construction not only for the negative parts of quantum groups but also for some well-known bases of these algebras. These constructions are simpler than those given in \cite{Ros98} and have more advantages: for example, we can deduce explicitly a commutation formula for PBW basis. We will return to these discussions in a consecutive paper.
\end{remark}

\end{document}